\documentclass[sts]{imsart}
\usepackage{latexsym}
\usepackage{graphicx}
\usepackage{grffile}  % for inserting image name with unusal character
\usepackage{mathrsfs} % for command \mathscr
\usepackage{amsmath,amssymb,amsthm}  % math packages
\usepackage[american]{babel}
\usepackage[colorlinks,citecolor=blue,urlcolor=blue]{hyperref}
\usepackage{bm}
\usepackage{dsfont}
\usepackage{subfigure}
\usepackage{xcolor}
\usepackage[T1]{fontenc}
\usepackage[utf8]{inputenc}
\usepackage{lmodern}
\usepackage{float}
\usepackage{caption}
\usepackage{mathtools}
\usepackage{color}
\usepackage{booktabs} % for nice looking tables
\usepackage{bbm}  % for mathbbm
\usepackage{algorithm}
\usepackage{algorithmic}

\usepackage{apptools}
\newcommand{\egaldef}{\stackrel{\Delta}{=}}

\usepackage[symbol]{footmisc}
\usepackage{comment}

\usepackage{custom_style}

\theoremstyle{definition}

\def\P{\mathds{P}}
\def\*#1{\mathbf{#1}}

\usepackage[left=1.4in, right=1.4in]{geometry}

\newtheorem{lemma}{Lemma} % Lemma 

\newtheorem{remark}{Remark}
\newtheorem{theorem}{Theorem}
\newtheorem{proposition}{Proposition}
\newtheorem{definition}{Definition}
\newtheorem{assumption}{Assumption}
\newtheorem{result}{Result}
\newcommand{\Comment}[1]{\textcolor{blue}{\textbf{\texttt{// #1}}}}
\newcommand{\paren}[1]{\left( #1 \right)} 
\newcommand{\set}[1]{\left\{ #1 \right\}}
\newcommand{\R}{\mathbb{R}}
\newcommand{\bhl}{\widehat{\bm\beta}_{\lambda}}
\newcommand{\fhl}{\widehat{f}_{\lambda}}
\newcommand{\Jhl}{\widehat{J}_{\lambda}}
\newcommand{\zh}{\widehat{\bz}}
\renewcommand{\P}{\mathbb{P}}

\newcommand\numberthis{\addtocounter{equation}{1}\tag{\theequation}}

\begin{document}
\begin{frontmatter}
\title{Aggregation of Multiple Knockoffs}
\runtitle{Aggregation of Multiple Knockoffs}
\date{}

\author[1,2]{Tuan-Binh Nguyen
  \thanks{Corresponding email: \textsc{tuan-binh.nguyen@inria.fr}}}
\author[2]{J\'er\^ome-Alexis Chevalier}\\
\author[2]{Bertrand Thirion}
\author[1]{Sylvain Arlot}

\runauthor{Nguyen et al.}

\affiliation[1]{Universit\'e Paris-Saclay, CNRS, Inria, Laboratoire de
  math\'ematiques d'Orsay, 91405, Orsay, France}
\affiliation[2]{Inria, CEA, Universit\'e Paris-Saclay, France}
% \affil[ ]{\textit {\{
% tuan-binh.nguyen,jerome-alexis.chevalier,bertrand.thirion,sylvain.arlot\}@inria.fr}}

\maketitle
\begin{abstract}
  We develop an extension of the knockoff inference procedure, introduced by
  Barber and Cand\`es [2015].
  This new method, called aggregation of multiple knockoffs (AKO), addresses
  the instability inherent to the random nature of knockoff-based inference.
  Specifically, AKO improves both the stability and power compared with the
  original knockoff algorithm while still maintaining guarantees for false
  discovery rate control.
  We provide a new inference procedure, prove its core properties, and
  demonstrate its benefits in a set of experiments on synthetic and real
  datasets.
  \footnote{Code is available at:
    \href{https://github.com/ja-che/hidimstat}{github.com/ja-che/hidimstat}}
\end{abstract}

\end{frontmatter}
%%%%%%%%%%%%%%%%%%%%%%%%%%%%%%%%%%%%%%%%%%%%%%%%%%%%%%%%%%%%%%%%%%%%%%%%%%%%%%%%

\section{Introduction}

In many fields, multivariate statistical models are used to \textit{fit} some
outcome of interest through a combination of measurements or features.
For instance, one might predict the likelihood for individuals to declare a
certain type of disease based on genotyping information.
Besides prediction accuracy, the inference problem consists in defining which
measurements carry useful features for prediction.
More precisely, we aim at conditional inference (as opposed to marginal
inference), that is, analyzing which features carry information \emph{given} the
other features.
This inference is however very challenging in high-dimensional settings.

Among the few available solutions, knockoff-based (KO) inference
\cite{barber_controlling_2015,candes_panning_2018} consists in introducing noisy
copies of the original variables that are independent from the outcome
conditional on the original variables, and comparing the coefficients of the
original variables to those of the knockoff variables.
This approach is particularly attractive for several reasons:
\textit{i)} it is not tied to a given statistical model, but can work instead
for many different multivariate functions, whether linear or not;
\textit{ii)} it requires a good generative model for features, but poses few
conditions for the validity of inference; and
\textit{iii)} it controls the false discovery rate (FDR,
\cite{benjamini_controlling_1995}), a more useful quantity than
multiplicity-corrected error rates.

Unfortunately, KO has a major drawback, related to the random nature of the
knockoff variables: two different draws yield two different solutions, leading
to large, uncontrolled fluctuations in power and false discovery proportion
across experiments (see Figure~\ref{fig:histo-fdp-power} below).
This makes the ensuing inference irreproducible.
An obvious way to fix the problem is to rely on some type of statistical
aggregation, in order to consolidate the inference results.
Such procedures have been introduced by \cite{gimenez_improving_2019} and by 
\cite{emery_controlling_2019}, but they have several limitations: the
computational complexity scales poorly with the number $B$ of bootstraps, while
the power of the method decreases with $B$.
In high-dimensional settings that we target, these methods are thus only
usable with a limited number of bootstraps.

In this work, we explore a different approach, that we call aggregation of
multiple knockoffs (AKO): it rests on a reformulation of the original knockoff
procedure that introduces intermediate p-values.
As it is possible to aggregate such quantities even without assuming
independence \cite{meinshausen_p-values_2009}, we propose to perform
aggregation at this intermediate step.
We first establish the equivalence of AKO with the original knockoff
aggregation procedure in case of one bootstrap (Proposition~\ref{prop:equivalence}).
Then we show that the FDR is also controlled with AKO (Theorem~\ref{thm:fdr-control}).
By construction, AKO is more stable than (vanilla) knockoff; we also
demonstrate empirical benefits in several examples, using simulated data, but
also genetic and brain imaging data.
Note that the added knockoff generation and inference steps are embarrassingly
parallel, making this procedure no more costly than the original KO inference.

\textbf{Notation.} Let \( [p] \) denote the set \( \{1, 2, \dots, p \} \);
for a given set given set \( \cA \),
\( | \cA | \egaldef \mathbf{card}(\cA)\);
matrices are denoted in bold uppercase letter, while vectors in bold
lowercase letter and scalars normal character. 
An exception for this is the vector of knockoff statistic \( \bW \), in which we
follow the notation from the original paper of \cite{barber_controlling_2015}.

%%%%%%%%%%%%%%%%%%%%%%%%%%%%%%%%%%%%%%%%%%%%%%%%%%%%%%%%%%%%%%%%%%%%%%%%%%%%%%%%

\section{Background}

\textbf{Problem Setting.} Let \( \bX \in \bbR^{n \times p} \) be a design
matrix corresponding to \( n \) observations of \( p \) potential explanatory
variables \( \bx_1, \bx_2, \dots, \bx_n \in \bbR^p \), with its target vector
\( \by \in \bbR^n \).
To simplify the exposition, we focus on sparse linear models, as 
\cite{barber_controlling_2015} and \cite{candes_panning_2018}: 
\begin{equation}
  \label{eq:linear}
  \by = \bX \bm\beta^* + \sigma\bm\epsilon
\end{equation}
where \( \bm\beta^* \in \bbR^p \) is the true parameter vector,
\( \sigma \in \bbR^+\) the unknown noise magnitude, \( \bm\epsilon \in \bbR^n\)
some Gaussian noise vector.
%
%A proportion of the variables belongs to the set of null features \( \cH_0 = \{j
%\in [p]: \beta_j^* = 0\} \), or equivalently \( j \in \cH_0 \iff \bx_j \perp \by
% \ | \ \bx_{-j}\).
%
%
Yet, it should be noted that the algorithm does not require linearity or
sparsity.
Our main interest is in finding an estimate \( \widehat{\cS} \) of the true support
set \( \cS = \{j \in [p]: \beta_j^* \neq 0 \} \), or the set of important
features that have an effect on the response.
As a consequence, the complementary of the support \( \cS \), which is 
denoted \( \cS^c = \{j \in [p]: \beta_j^* = 0 \} \), corresponds to null
hypotheses.
Identifying the relevant features amounts to simultaneously testing
\[
  \cH_0^j: \beta_j^* = 0 \quad \text{versus} \quad \cH_{a}^j: \beta_j^* \neq 0,
  \quad \forall j = 1, \dots, p.
\]
Specifically, we want to bound the proportion of false positives among selected
variables, that is, control the false discovery rate (FDR,
\cite{benjamini_controlling_1995}) under certain predefined level~$\alpha$:
\[
  \text{FDR} = \bbE \left[\dfrac{\ | \widehat{\cS} \cap \cS^c |}{| \widehat{\cS} | \vee
      1} \ \right] \leq \alpha \in (0, 1)
\, . 
\]

\textbf{Knockoff Inference.} Introduced originally by
\cite{barber_controlling_2015}, the knockoff filter is a variable selection
method for multivariate models with theoretical control of FDR.
\cite{candes_panning_2018} expanded the method to work in the case of
(mildly) high-dimensional data, with the assumption that \( \bx =
(x_1, \dots, x_p) \sim P_X \) such that \( P_X \) is known.
The first step of this procedure involves sampling extra null variables that
have a correlation structure similar to that of the original variables, with the
following formal definition.
\begin{definition}[Model-X knockoffs, \cite{candes_panning_2018}]
  \label{defn:knockoff}
  The model-X knockoffs for the family of random variables
  \( \bx = (x_1, \dots, x_p) \) are a new family of random variables
  \( \tilde{\bx} = (\tilde{x}_1, \dots, \tilde{x}_p) \) constructed to satisfy
  the two properties\textup{:}
\begin{enumerate}
\item For any subset \( \cK \subset \discset{1, \dots, p} \), 
  \((\bx, \tilde{\bx})_{\text{swap}(\cK)} \stackrel{d}{=} (\bx, \tilde{\bx}) \),
  where the vector \( (\bx, \tilde{\bx})_{\text{swap}(\cK)} \) denotes the swap
  of entries \( x_j \) and \( \tilde{x}_j \) for all \( j \in \cK \),
  and \( \stackrel{d}{=} \) denotes equality in distribution. 
\item \( \tilde{\bx} \independent \by \mid \bx \, \).
\end{enumerate}
\end{definition}
A test statistic is then calculated to measure the strength of the original
variables versus their knockoff counterpart.
We call this the knockoff statistic \( \bW = \{W_j \}_{j=1}^{p} \),
that must fulfill two important properties.
\begin{definition}[Knockoff statistic, \cite{candes_panning_2018}]
  \label{defn:knockoff-statistic}
  A knockoff statistic \( \bW = \{W_{j}\}_{j \in [p]} \) is a measure of feature
  importance that satisfies the two following properties\textup{:}
  \begin{enumerate}
  \item It depends only on \( \bX, \tilde{\bX} \) and \( \by \)
    \[
      \bW = f (\bX, \tilde{\bX}, \by) 
      \, . 
    \]
  \item Swapping the original variable column \( \bx_j \) and its knockoff column
    \( \tilde{\bx}_j \) switches the sign of \( W_j \)\textup{:} 
    % iff \( j \) is in the support set \( \cS \):
    % Bertrand: swapping variables always switches the sign ...
    \begin{equation*}
      W_j([\bX, \tilde{\bX}]_{swap(S)}, y) =
      \left\{
          \begin{array}{ll}
            W_j([\bX, \tilde{\bX}], \by) \ \text{if } j \in \cS^c \\
            -W_j([\bX, \tilde{\bX}], y) \ \text{if } j \in \cS \, . \\
          \end{array}
      \right.
    \end{equation*}
  \end{enumerate}
\end{definition}
Following previous works on the analysis of the knockoff properties
\cite{arias-castro_distribution_2017,rabinovich_optimal_2020}, we make
the following assumption about the knockoff statistic.
This is necessary for our analysis of knockoff aggregation scheme later on.
\begin{assumption}[Null distribution of knockoff statistic]
  \label{assumption:ko-stat-dist} 
  The knockoff statistic defined in
  Definition~\ref{defn:knockoff-statistic} are such that 
  \( \{W_j\}_{j \in \cS^c}, \) are
  independent and follow the same distribution \( \bbP_0 \).
\end{assumption}
\begin{remark}
\label{remark:P0-symmetric}
As a consequence of \cite[Lemma 2]{candes_panning_2018} regarding the signs of
the null \( W_j \) as i.i.d. coin flips, if
Assumption~\ref{assumption:ko-stat-dist} holds true, then $\bbP_0$ is symmetric
around zero.
\end{remark}
One such example of knockoff statistic is the Lasso-coefficient difference
(LCD).
The LCD statistic is computed by first making the concatenation of original
variable and knockoff variables \( [\bX, \tilde{\bX}] \in \bbR^{n \times 2p}\),
then solving the Lasso problem \cite{tibshirani_regression_1996}:
\begin{equation}
  \label{eq:lasso-problem}
  \widehat{\bm\beta} = \argmin_{\bm\beta \in \bbR^{2p}} \left\{ \dfrac{1}{2}
  \norm{\by - [\bX, \tilde{\bX}] \bm\beta}_2^2 + \lambda \norm{\bm\beta}_1 \right\} 
\end{equation}
with \( \lambda \in \bbR \) the regularization parameter, and finally to  take:
\begin{equation}
  \label{eq:lcd}
\forall j \in [p] \, , \qquad 
   W_j = |\widehat{\beta}_j| - |\widehat{\beta}_{j + p}| \, .
\end{equation}
This quantity measures how strong the coefficient magnitude of each original
covariate is against its knockoff, hence the name Lasso-coefficient
difference. Clearly, the LCD statistic satisfies the two properties stated in
Definition~\ref{defn:knockoff-statistic}.

Finally, a threshold for controlling the FDR under given level
\( \alpha \in (0, 1) \) is calculated:
\begin{equation}
  \label{eq:ko-threshold}
  \tau_+ = \min \discset{t > 0: \dfrac{1 + \# \{j : W_j \leq -t \}}{\# \{j: W_j
      \geq t \} \vee 1 } \leq \alpha},
\end{equation}
and the set of selected variables is 
\( \widehat{\cS} = \{j \in [p] : W_j \geq \tau_+ \} \).
% \begin{remark}
%   A slightly modified version of the threshold without the offset additional
%   \( +1 \) in the numerator of the fraction in Eq.~\refp{eq:ko-threshold} was
%   also introduced in \cite{barber_controlling_2015} and
%   \cite{candes_panning_2018}:
%   \[
%     \tau = \min \discset{t > 0: \dfrac{\# \{j: W_j \leq -t \}}{\# \{j: W_j
%         \geq t \} \vee 1 } \leq \alpha}.
%   \]
%   However this threshold cannot strictly control FDR under level \( \alpha \in
%   (0, 1) \).
%   %
%   Henceforth, we will use KO as an abbreviation for the knockoff+ procedure.
% \end{remark}

%
\textbf{Instability in Inference Results.} Knockoff inference is a flexible
method for multivariate inference in the sense that it can use different
loss functions (least squares, logistic, etc.), and use different variable
importance statistics.
However, a major drawback of the method comes from the random nature of the
knockoff variables \( \tilde{\bX} \) obtained by sampling: different draws yield
different solutions (see Figure~\ref{fig:histo-fdp-power} in Section~\ref{ssec:synthetic-experiment}).
This is a major issue in practical settings, where knockoff-based
inference is used to prove the conditional association between
features and outcome.

%%%%%%%%%%%%%%%%%%%%%%%%%%%%%%%%%%%%%%%%%%%%%%%%%%%%%%%%%%%%%%%%%%%%%%%%%%%%%%%%

\section{Aggregation of Multiple Knockoffs}

\subsection{Algorithm Description}

One of the key factors that lead to the extension of the original (vanilla) knockoff filter stems from
the observation that knockoff inference can be formulated based on the following
quantity.
\begin{definition}[Intermediate p-value]
  Let \( \bW = \{W_{j}\}_{j \in [p]} \) be a knockoff statistic
  according to Definition~\ref{defn:knockoff-statistic}. For \( j = 1, \dots, p \), the
  intermediate p-value \( \pi_j \) is defined as\textup{:}
\begin{equation}\label{eq:pval}
  \pi_j = %
    \begin{cases}
      \dfrac{1 + \# \{k: W_k \leq -W_j \}}{p} \quad \text{if} \quad W_j > 0\\
      1 \quad \text{if} \quad W_j \leq 0 
      \, . 
    \end{cases}
\end{equation}
\label{defn:pval}
\end{definition}

We first compute \( B \) draws of knockoff variables, and then
knockoff statistics.
Using Eq.~\Cref{eq:pval}, we derive the corresponding empirical p-values
\(\pi_j^{(b)} \), for all $j \in [p]$ and $b \in [B]$.
Then, we aggregate them for each variable \( j \) in parallel, using the
quantile aggregation procedure introduced by \cite{meinshausen_p-values_2009}:
\begin{equation}
  \label{eq:q-agg}
  \bar{\pi}_j = \min \left\{1, \dfrac{q_{\gamma} \bigl( \{ \pi_j^{(b)} : b \in [B] \} \bigr)}{\gamma} \right\}
\end{equation}
where \( q_{\gamma}(\cdot) \) is the \( \gamma\)-quantile function.
In the experiments, we fix \( \gamma=0.3 \) and \( B=25 \).
The selection of these default values is explained more thoroughly in Section
\ref{ssec:synthetic-experiment}.

Finally, with a sequence of aggregated p-values
\( \bar{\pi}_1, \dots, \bar{\pi}_p \), we use Benjamini-Hochberg step-up
procedure (BH, \cite{benjamini_controlling_1995}) to control the FDR.
\begin{definition}[BH step-up, \cite{benjamini_controlling_1995}]
  \label{def:bh-step-up}
  Given a list of p-values \( \bar{\pi}_1, \dots, \bar{\pi}_p \) and predefined
  FDR control level \( \alpha \in (0, 1) \), the Benjamini-Hochberg step-up
  procedure comprises three steps\textup{:}
\begin{enumerate}
\item Order p-values such that:
  \( \bar{\pi}_{(1)} \leq \bar{\pi}_{(2)} \leq \dots \leq \bar{\pi}_{(p)} \).
\item Find:
  \begin{equation}\label{eq:pval-threshold}
    \widehat{k}_{BH} = \max \left\{ k: \bar{\pi}_{(k)} \leq \dfrac{k\alpha}{p} \right\} 
    \, . 
  \end{equation}  
\item Select
  \( \widehat{\cS} = \{j \in [p]: \bar{\pi}_{(j)} \leq \bar{\pi}_{(\widehat{k}_{BH})} \}
  \). 
\end{enumerate}
\end{definition}
This procedure controls the FDR, but only under independence or
positive-dependence between p-values \cite{benjamini_control_2001}.
As a matter of fact, for a strong guarantee of FDR control, one can consider
instead a threshold yielding a theoretical control of FDR under arbitrary
dependence, such as the one of \cite{benjamini_control_2001}.
We call BY step-up the resulting procedure.
Yet we use BH step-up procedure in the experiments of Section
\ref{sec:experiments}, as we observe empirically that the aggregated p-values
\( \bar{\pi}_j \) defined in Eq.~\eqref{eq:pval} does not deviate significantly
from independence (details in Appendix).
\begin{definition}[BY step-up, \cite{benjamini_control_2001}]
  \label{def:BY-step-up}
  Given an ordered list of p-values as in step 1 of BH step-up
  \( \bar{\pi}_{(1)} \leq \bar{\pi}_{(2)} \leq \cdots \leq \bar{\pi}_{(p)} \)
  and predefined level \( \alpha \in (0, 1) \), the Benjamini-Yekutieli step-up
  procedure first finds\textup{:}
  \begin{equation}\label{eq:pval-threshold-BY}
    \widehat{k}_{BY} = \max \left\{ k \in [p] : \bar{\pi}_{(k)}
      \leq \dfrac{k\beta(p)\alpha}{p} \right\},
  \end{equation}
  with \( \beta(p) = (\sum_{i=1}^p 1 / i)^{-1} \), 
  and then selects 
  \[ \widehat{\cS} = \bigl\{ j \in [p]: \bar{\pi}_{(j)} \leq \bar{\pi}_{(\widehat{k}_{BY})} \bigr\}
  \, .  \]
\end{definition}
\cite{blanchard_adaptive_2009} later on introduced a general function form for
\( \beta(p) \) to make BY step-up more flexible.
However, because we always have \( \beta(p) \leq 1 \), this procedure leads
to a smaller threshold than BH step-up, thus being more conservative.

\begin{algorithm}[!]
\caption{AKO -- Aggregation of multiple knockoffs}
\label{alg:ako}
\begin{algorithmic}
\STATE {\bfseries Input:} \( \bX \in \bbR^{n \times p}\), \( \by \in \bbR^{n} \),
\( B \) -- number of bootstraps ;  \( \alpha \in (0, 1) \) -- target FDR level \\
\STATE {\bfseries Output:} \( \widehat{S}_{AKO} \) -- Set of selected variables
index

\vspace{0.5em}

\FOR{\(b=1\) \textbf{to} \(B\)}
  \STATE \( \tilde{\bX}^{(b)} \gets \textsc{sampling\_knockoff}(\bX) \)
  \STATE \( \bW^{(b)} \gets \textsc{knockoff\_statistic}(\bX, \tilde{\bX}^{(b)}, \by) \)

  \STATE \( \bm\pi^{(b)} \gets \textsc{convert\_statistic}(\bW^{(b)}) \)
  \Comment{Using Eq.~\refp{eq:pval}}  
\ENDFOR

\vspace{0.5em}

\FOR{\( j=1 \) \textbf{to} \( p \) }
\STATE \( \bar{\pi}_j \gets
\textsc{quantile\_aggregation}\left(\{\pi_j^{(b)}\}^B_{b=1}\right) \)
\Comment{Using Eq.~\refp{eq:q-agg}}
\ENDFOR

\vspace{0.5em}

\STATE \( \widehat{k} \gets
  \textsc{fdr\_threshold}(\alpha, \left(\bar{\pi}_1, \bar{\pi}_2, \dots,
  \bar{\pi}_p) \right) \)
  \Comment{Using either Eq. \refp{eq:pval-threshold} or
    Eq.~\refp{eq:pval-threshold-BY}}
  
\vspace{0.5em}
  
\STATE \textbf{Return:} \( \widehat{S}_{AKO} \gets \discset{j \in [p]: \bar{\pi}_j \leq
    \bar{\pi}_{\widehat{k}} } \)

\end{algorithmic}
\end{algorithm}

The AKO procedure is summarized in Algorithm \ref{alg:ako}.
We show in the next section that with the introduction of the aggregation step,
the procedure offers a guarantee on FDR control under mild hypotheses.
Additionally, the numerical experiments of Section \ref{sec:experiments}
illustrate that aggregation of multiple knockoffs indeed improves the stability
of the knockoff filter, while bringing significant statistical power gains.

\subsection{Related Work}
\label{ssec:related-works}

To our knowledge, up until now there have been few attempts to stabilize
knockoff inference.
Earlier work of \cite{su_communication_2015} rests on the same idea of
generating multiple knockoff bootstrap as ours, but relies on the linear
combination of the so-called \emph{one-bit p-values} (introduced as a means to
prove the FDR control in original knockoff work of
\cite{barber_controlling_2015}).
As such, the method is less flexible since it requires a specific type of
knockoff statistic to work.
Furthermore, it is unclear how this method would perform in high-dimensional
settings, as it was only demonstrated in the case of \( n > p \).
More recently, the work of \cite{holden_multiple_2018} incorporates directly
multiple bootstraps of knockoff statistics for FDR thresholding without the
need of p-value conversion.
Despite its simplicity and convenience as a way of aggregating knockoffs, our
simulation study in Section~\ref{ssec:synthetic-experiment} demonstrates that
this method somehow fails to control FDR in several settings.

In a different direction, \cite{gimenez_improving_2019} and
\cite{emery_controlling_2019} have introduced \emph{simultaneous knockoff}
procedure, with the idea of sampling several knockoff copies at the same time
instead of doing the process in parallel as in our work.
This, however, induces a prohibitive computational cost when the number of
bootstraps increases, as opposed to the AKO algorithm that can use parallel
computing to sample multiple bootstraps at the same time. 
In theory, on top of the fact that sampling knockoffs has cubic complexity on
runtime with regards to number of variables \( p \) (requires covariance matrix
inversion), simultaneous knockoff runtime is of \( \cO(B^3 p^3) \), while for
AKO, runtime is only of \( \cO(Bp^3) \) and \( \cO(p^3) \) with parallel
computing.
Moreover, the FDR threshold of simultaneous knockoff is calculated in such a way
that it loses statistical power as the number of bootstraps increases, when the
sampling scheme of vanilla knockoff by \cite{barber_controlling_2015} is used.
We have set up additional experiments in the Appendix to illustrate this
phenomenon.
In addition, the threshold introduced by \cite{emery_controlling_2019} is only
proven to have a theoretical control of FDR in the case where $n > p$.

%% Appendix \ref{sup-ako-vs-sko}

%%%%%%%%%%%%%%%%%%%%%%%%%%%%%%%%%%%%%%%%%%%%%%%%%%%%%%%%%%%%%%%%%%%%%%%%%%%%%%%%
\section{Theoretical Results}
\label{sec:results}

We now state our theoretical results about the AKO procedure.
\subsection{Equivalence of Aggregated Knockoff with Single Bootstrap}
  % ($\mathbf{B=1, \bm\gamma=1}$) and Vanilla Knockoff}
%
First, when $B=1$ and $\gamma=1$, we show that AKO+BH is equivalent to vanilla
knockoff.
\begin{proposition}[Proof in Appendix~\ref{ssec:proof-prop-equiv}]
  \label{prop:equivalence}
  Assume that for all \( j, j' = 1, \dots, p \), 
  \[
    \bbP (W_j = W_{j'}, \quad W_j \neq 0, \quad W_{j'} \neq 0) = 0
  \]
  that is, non-zero LCD statistics are distinct with probability~1. 
  Then, single bootstrap version of aggregation of multiple knockoffs \textup{(}$B=1$\textup{)}, using
  \( \gamma=1 \) and BH step-up procedure in Definition~\ref{def:bh-step-up} for
  calculating FDR threshold, is equivalent to the original knockoff inference by
  \cite{barber_controlling_2015}.
\end{proposition}

\begin{remark}
  Although Proposition~\ref{prop:equivalence} relies on the assumption of distinction
  between non-zero \( W_j \)s for all \( j = 1, \dots, p \), the following lemma
  establishes that this assumption holds true with probability one for the LCD statistic up to further
  assumptions.% (found in Appendix \ref{sup-ko-distinct}).
\end{remark}

\begin{lemma}[Proof in Appendix~\ref{ko-distinct}]
  \label{lemma:ko-distinct}
  Define the equi-correlation set as\textup{:}
  \[
    \widehat{J}_{\lambda} = \discset{j \in [p]: \bx_j^{\top} (\by - \bX
      \widehat{\bm\beta}) = \lambda/ 2 }
  \]
  with \( \widehat{\bm\beta}, \lambda \) defined in Eq.~\refp{eq:lasso-problem}.
  Then we have\textup{:}
  \begin{equation}
    \label{eq:ko-distinct}
    \bbP\left( W_j = W_{j'}, W_j \neq 0, W_{j'} \neq 0, \
      \rank(X_{\widehat{J}_\lambda}) = |\widehat{J}_\lambda| \right) = 0 
  \end{equation}
  for all \( j, j' \in [p]: j \neq j' \).
  In other words, assuming \( \bX_{\widehat{J}_{\lambda}} \) is full rank, then the
  event that LCD statistic defined in Eq.~\refp{eq:lcd} is distinct for all
  non-zero value happens almost surely.
  
\end{lemma}

\subsection{Validity of Intermediate P-values}
Second, the fact that the $\pi_j$ are called ``intermediate p-values'' is
justified by the following lemma.
\begin{lemma}
  \label{lemma:pval}
  If Assumption~\ref{assumption:ko-stat-dist} holds true, and if
  $\lvert \cS^c \rvert \geq 2$, then, for all \( j \in \cS^c \), the
  intermediate p-value \( \pi_j \) defined by Eq.~\refp{eq:pval}
  satisfies\textup{:}
  \begin{equation*}
    \forall t \in [0,1] \qquad \bbP(\pi_j \leq t)
    \leq \frac{\kappa p}{\lvert \cS^c \rvert} t
  \end{equation*}
  %
  % \vspace{-10mm}
  where $\kappa = \dfrac{\sqrt{22}-2}{7 \sqrt{22}-32} \leq 3.24$. 
\end{lemma}
\begin{proof}
  \label{proof-lemma:pval}
  The result holds when \( t \geq 1 \) since
  $\kappa p \geq p \geq \lvert \cS^c \rvert$ and a probability is always
  smaller than~$1$.
  Let us now focus on the case where \( t \in [0, 1) \), and define
  $m = \lvert \cS^c \rvert - 1 \geq 1$ by assumption.
  Let \( F_0 \) denote the c.d.f. of \( \bbP_0 \), the common distribution of
  the null statistics \(\discset{W_k}_{k \in \cS^c} \), which exists by
  Assumption~\ref{assumption:ko-stat-dist}.
  Let \( j \in \cS^c \) be fixed.
  By definition of \( \pi_j \), when \( W_j > 0 \) we have:
  \begin{align*}
    \pi_j &= \dfrac{1 + \#\{k \in [p]: W_k \leq -W_j\}}{p}
    \\ &= \dfrac{1 + \#\{k \in \cS: W_k \leq -W_j\}}{p}
         + \dfrac{\#\{k \in \cS^c\setminus\{ j \}: W_k \leq -W_j\}}{p}
    \\ & \qquad \text{(since \( W_j > 0 > - W_j \))} 
    \\ &\geq \frac{m}{p} \widehat{F}_{m} (-W_j) + \frac{1}{p}
           \numberthis \label{eq:pr.le.p-value.non-asympt.minor-pi-j} 
  \end{align*}
  where $\forall u \in \bbR$,
  $\widehat{F}_{m} (u) \egaldef \dfrac{\#\{ k \in \cS^c\setminus\{ j \}: W_k
    \leq u \} }{m}$ is the empirical cdf of
  \( \{ W_k\}_{k \in \setminus \{j\}} \).
  Therefore, for every $t \in [0,1)$,
  \begin{align*}
    \bbP(\pi_j \leq t)
    &= \bbP(\pi_j \leq t \text{ and } W_j > 0 )
      + \underbrace{\bbP(\pi_j \leq t \ \text{and} \ W_j \leq 0) }_{=0 
      \text{ since } \pi_j = 1 \text{ when } W_j \leq 0}  
    \\ &=  \bbE \bigl[ \bbP( \pi_j \leq t \mid W_j) \mathbbm{1}_{W_j > 0 } \bigr] 
    \\ &\leq \bbE \left[ \bbP \left(  \frac{m}{p} \widehat{F}_{m} (-W_j)
         + \frac{1}{p} \leq t \mid W_j \right) \mathbbm{1}_{W_j > 0 } \right] 
         \ \text{by ~\eqref{eq:pr.le.p-value.non-asympt.minor-pi-j}} 
    \\ 
    &\leq \bbP \left(  \frac{m}{p} \widehat{F}_{m} (-W_j)    + \frac{1}{p} \leq t \right) 
      \, . 
      \numberthis\label{eq:pr.le.p-value.non-asympt.major-1}
  \end{align*}
  Notice that $W_j$ has a symmetric distribution around 0, as shown by
  Remark~\ref{remark:P0-symmetric}, that is, $-W_j$ and $W_j$ have the same
  distribution.
  Since $W_j$ and  $\{W_k\}_{k \in \cS^c \backslash\{j\}} $ are independent with the
  same distribution $\bbP_0$ by Assumption~\ref{assumption:ko-stat-dist}, they
  have the same joint distribution as
  $F_0^{-1}(U), F_0^{-1}(U_1), \ldots, F_0^{-1}(U_m)$ where
  $U,U_1, \ldots, U_m$ are independent random variables with uniform
  distribution over $[0,1]$, and $F_0^{-1}$ denotes the generalized inverse of
  $F_0$.
  Therefore, Eq.~\eqref{eq:pr.le.p-value.non-asympt.major-1} can be rewritten
  as
  \begin{gather}
    \bbP(\pi_j \leq t)
    \leq \bbP \left(  \frac{m}{p} \widetilde{F}_{m} \bigl( F_0^{-1}(U) \bigr)  + \frac{1}{p} \leq t \right) 
    \label{eq:pr.le.p-value.non-asympt.major-2}
    \\ 
    \text{where} \qquad 
    \forall v \in \bbR, \qquad 
    \widetilde{F}_{m}(v) 
    \egaldef \frac{1}{m} \sum_{k=1}^m \mathbbm{1}_{ F_0^{-1}(U_k) \leq v }
    \notag 
    \, . 
  \end{gather}
  Notice that for every $u \in \bbR$, 
  \begin{flalign*}
    \widehat{G}_{m}(u) 
    &\egaldef \frac{1}{m} \sum_{k=1}^m \mathbbm{1}_{ U_k \leq u }
    \leq \frac{1}{m} \sum_{k=1}^m \mathbbm{1}_{ F_0^{-1}(U_k) \leq F_0^{-1}(u) } \\
    &= \widetilde{F}_{m} \bigl( F_0^{-1}(u) \bigr)
  \end{flalign*}
  since $F_0^{-1}$ is non-decreasing. 
  Therefore, Eq.~\eqref{eq:pr.le.p-value.non-asympt.major-2} shows that
  \begin{align}
    \notag
    \bbP(\pi_j \leq t)
    &\leq 
      \bbP \left( m \widehat{G}_{m}(U) \leq tp - 1 \right) \\
    &= \int_0^1 \bbP \left( m \widehat{G}_{m}(u) \leq tp - 1 \right) \mathrm{d}u 
      \, . 
      \label{eq:pr.le.p-value.non-asympt.major-3}
  \end{align}

  Now, we notice that for every $u \in (0,1)$, $m \widehat{G}_{m}(u)$ follows a
  binomial distribution with parameters $(m,u)$.
  So, a standard application of Bernstein's inequality
  \cite[Eq. 2.10]{Bou_Lug_Mas:2011:livre} shows that for every
  $0 \leq x \leq u \leq 1$,
  \begin{align*}
    \bbP \left( m \widehat{G}_{m}(u) \leq mx \right) 
    &\leq \exp\left( \frac{-m^2 (u-x)^2}{2 mu + \frac{m (u-x)}{3}} \right) \\
    % 
    %% = \exp\left( \frac{- 3 m (u-x)^2}{6x+ 7 (u-x)} \right) 
    % 
    &= \exp\left( \frac{- 3 m x \left( \frac{u}{x} - 1 \right)^2}{\frac{7 u}{x} - 1} \right) 
      . 
  \end{align*}
  Note that for every $\lambda \in (0,1/7)$, we have 
  \[
    \forall w \geq \frac{1-\lambda}{1-7\lambda} \geq 1 \, , 
    \qquad 
    \frac{w-1}{7w-1} \geq \lambda 
  \]
  hence $\forall u \geq x \dfrac{1-\lambda}{1-7\lambda}$,
  \[
    \bbP \left( m \widehat{G}_{m}(u) \leq mx \right) 
    \leq \exp\left[ - 3 m \lambda x \left( \frac{u}{x} - 1 \right) \right] \, . 
  \]
  As a consequence, $\forall \lambda \in ( 0 , 1/7 )$,
  \begin{align*}
    \int_0^1 \bbP \left( m \widehat{G}_{m}(u) \leq mx \right) \mathrm{d}u
    &\leq 
      \frac{1-\lambda}{1-7\lambda} x +
      \int_{\frac{1-\lambda}{1-7\lambda} x}^1 \exp\bigl[ - 3 m \lambda (u-x)
      \bigr] \mathrm{d}u 
    \\&\leq 
    \frac{1-\lambda}{1-7\lambda} x +
    \int_{\frac{6\lambda}{1-7\lambda} x}^{+\infty} \exp ( - 3 m \lambda v  ) \mathrm{d}v 
    \\&\leq 
    \frac{1-\lambda}{1-7\lambda} x +
    \frac{1}{3 m\lambda} \exp\left(- 3 m \lambda \frac{6\lambda}{1-7\lambda} x \right)
    \\ &\leq \frac{1-\lambda}{1-7\lambda} x + \frac{1}{3 m\lambda}. 
  \end{align*}
  Taking $x = (tp-1)/m$, we obtain from
  Eq.~\eqref{eq:pr.le.p-value.non-asympt.major-3} that
  $\forall \lambda \in ( 0, 1/7 )$
  \begin{align*}
    \bbP(\pi_j \leq t)
    &\leq 
      \frac{1-\lambda}{1-7\lambda} \frac{tp-1}{m} + \frac{1}{3m\lambda}
    \\&= \frac{1-\lambda}{1-7\lambda} \frac{tp}{m} 
      + \left( \frac{1}{3\lambda} - \frac{1-\lambda}{1-7\lambda} \right) \frac{1}{m} 
      \, .     \numberthis \label{eq:kappa-value}
  \end{align*}
  Choosing $\lambda = (5 - \sqrt{22})/3 \in (0,1/7)$, we have
  $\frac{1}{3\lambda} = \frac{1-\lambda}{1-7\lambda}$ hence the result with
  \[ 
    \kappa = \frac{1-\lambda}{1-7\lambda}
    = \frac{\sqrt{22}-2}{7 \sqrt{22}-32} 
    \leq 3.24.
  \]
\end{proof}

\begin{remark}
If the definition of $\pi_j$ is replaced by  
\begin{equation} %%\label{eq:pval:variant}
  \pi_{j,c} \egaldef  %
    \begin{cases}
      \dfrac{c + \# \{k: W_k \leq -W_j \}}{p} \quad \text{if} \quad W_j > 0\\
      1 \quad \text{if} \quad W_j \leq 0 
    \end{cases}
\end{equation}
for some $c > 0$, the above proof also applies and yields an upper bound of the
form
\[ 
\forall t \geq 0 \, , \qquad 
\bbP(\pi_{j,c} \leq t) \leq \kappa(c) t 
\]
for some constant $\kappa(c)>0$.
It is then possible to make $\kappa(c)$ as close to~$1$ as desired, by choosing
$c$ large enough.  Lemma~\ref{lemma:pval} corresponds to the case $c=1$.
\end{remark}

Note that we also prove in the Appendix that if $p \to +\infty$ with
$\abs{\cS} \ll p$, then for every $j \geq 1$ such that $\beta^*_j=0$, $\pi_j$
is an asymptotically valid p-value, that is,
\begin{equation} \label{eq.pval.asympt}
\forall t \in [0,1] \, ,  \qquad 
\limsup_{p \to +\infty}    \bbP(\pi_j \leq t) \leq t 
\, . 
\end{equation}
Yet, proving our main result (Theorem~\ref{thm:fdr-control}) requires a non-asymptotic bound such
that the one of Lemma~\ref{lemma:pval}.

\subsection{FDR control for AKO}
Finally, the following theorem provides a non-asymptotic guarantee about the FDR
of AKO with BY step-up.
\begin{theorem}%[Proof in Appendix \ref{sup-proof-thm:fdr-control}]
  \label{thm:fdr-control}
  If Assumption~\ref{assumption:ko-stat-dist} holds true and
  $\lvert \cS^c \rvert \geq 2$, then for any \( B \geq 1 \) and
  \( \alpha \in (0,1) \),
  the output \( \widehat{\cS}_{AKO+BY} \) of aggregation of multiple knockoff
  \textup{(}Algorithm \ref{alg:ako}\textup{)},   
  with the BY step-up procedure, has a FDR controlled as follows\textup{:}
  \[
    \bbE \left[ \dfrac{\ |\widehat{S}_{AKO+BY} \cap \ \cS^c| \ }{ |\widehat{S}_{AKO+BY}| \vee 1
      }\right] \leq \kappa \alpha
  \]
  where \( \kappa %= \dfrac{\sqrt{22}-2}{7 \sqrt{22}-32}
  \leq 3.24 \) is defined
  in Lemma~\ref{lemma:pval}.
\end{theorem}
\begin{proof}[Sketch of the proof]
  The proof of \cite[Theorem~3.3]{meinshausen_p-values_2009}, which itself
  relies partly on \cite{benjamini_control_2001}, can directly be adapted to
  upper bound the FDR of $\widehat{S}_{AKO+BY}$ in terms of quantities of the
  form $\bbP(\pi_j^{(b)} \leq t)$ for $j \in \cS^c$ and several $t \geq 0$.
  Combined with Lemma \ref{lemma:pval}, this yields the result.  A full proof is
  provided in Appendix~\ref{sec.pr.ThmPpal}.
\end{proof}
Note that Theorem~\ref{thm:fdr-control} loses a factor $\kappa$ compared to the
nominal FDR level $\alpha$.
This can be solved by changing $\alpha$ into $\alpha/\kappa$ in the definition
of $\widehat{S}_{AKO+BY}$.
Nevertheless, in our experiments, we do not use this correction and find that
the FDR is still controlled at level~$\alpha$.

%%%%%%%%%%%%%%%%%%%%%%%%%%%%%%%%%%%%%%%%%%%%%%%%%%%%%%%%%%%%%%%%%%%%%%%%%%%%%%%%

\section{Experiments}
\label{sec:experiments}
\textbf{Compared Methods.}
We make benchmarks of our proposed method aggregation of multiple knockoffs
(AKO) with \( B=25, \gamma=0.3 \) and vanilla knockoff (KO), along with other
recent methods for controlling FDR in high-dimensional settings, mentioned in
Section \ref{ssec:related-works}:
\emph{simultaneous knockoff}, an alternative aggregation scheme for
knockoff inference introduced by \cite{gimenez_improving_2019} (KO-GZ),
along with its variant of \cite{emery_controlling_2019} (KO-EK);
the \emph{knockoff statistics aggregation} by \cite{holden_multiple_2018} (KO-HH);
and \emph{debiased Lasso} (DL-BH) \cite{javanmard_fdr_2019}.
%
%%%%% Sylvain: I don't see why this sentence is useful here (in a paragraph about the *methods* compared)
%%The benchmarks are done on both synthetic and real scientific datasets, a
%%Genome Wide Association Studies (GWAS) and a functional Magnetic Resonance
%%Imaging (fMRI) dataset.

\subsection{Synthetic Data}
\label{ssec:synthetic-experiment}
\textbf{Simulation Setup.} Our first experiment is a simulation scenario where a
design matrix \( \bX \) (\( n=500, p=1000 \)) with its continuous response
vector \( \by \) are created following a linear model assumption.
The matrix is sampled from a multivariate normal distribution of zero mean and
covariance matrix \( \mathbf{\Sigma} \in \bbR^{p \times p} \).
We generate \( \mathbf{\Sigma} \) as a symmetric Toeplitz matrix that has the
structure:

\[
  \mathbf{\Sigma} =
  \begin{bmatrix}
    \rho^0  & \rho^{1} & \dots & \rho^{p-1} \\
    \rho^1  & \ddots & \dots &  \rho^{p-2} \\
    \vdots & \dots & \ddots & \vdots \\
    \rho^{p-1} & \rho^{p-2} & \dots & \rho^0
  \end{bmatrix}
\]

where the \( \rho \in (0, 1) \) parameter controls the correlation structure of
the design matrix. This means that neighboring variables are strongly
correlated to each other, and the correlation decreases with the distance between indices.
The true regression coefficient \( \bm\beta^* \) vector is picked with a
sparsity parameter that controls the proportion of non-zero elements with
amplitude 1.
The noise \( \bm\epsilon \) is generated to follow \( \cN(\bm\mu,
\bI_n) \) with its magnitude \( \sigma = \norm{\bX\bm\beta^*}_2 /
(\text{SNR}\norm{\bm\epsilon}_2) \) controlled by the SNR parameter.
The response vector \( \by \) is then sampled according to
Eq.~\refp{eq:linear}. In short, the three main parameters controlling
this simulation are correlation \( \rho \), sparsity degree \( k \)
and signal-to-noise ratio SNR.

\textbf{Aggregation Helps Stabilizing Vanilla Knockoff.}
To demonstrate the improvement in stability of the aggregated knockoffs, we first
do multiple runs of AKO and KO with $\alpha = 0.05$ under \emph{one simulation} of \( \bX \) and
\( \by \).
In order to guarantee a fair comparison, we compare 100 runs of AKO, each with
\( B=25 \) bootstraps, with the corresponding 2500 runs of KO.
We then plot the histogram of FDP and power in Figure~\ref{fig:histo-fdp-power}.
For the original knockoff, the false discovery proportion varies widely and has
a small proportion of FDP above $0.2=4\alpha$.
Besides, a fair amount of KO runs returns null power.

On the other hand, AKO not only improves the stability in the result for FDP
---the FDR being controlled at the nominal level $\alpha=0.05$--- but it also improves
statistical power: in particular, it avoids catastrophic behavior (zero power)
encountered with KO.

\begin{figure}[h]
  \centering
  \includegraphics[width=0.7\textwidth]{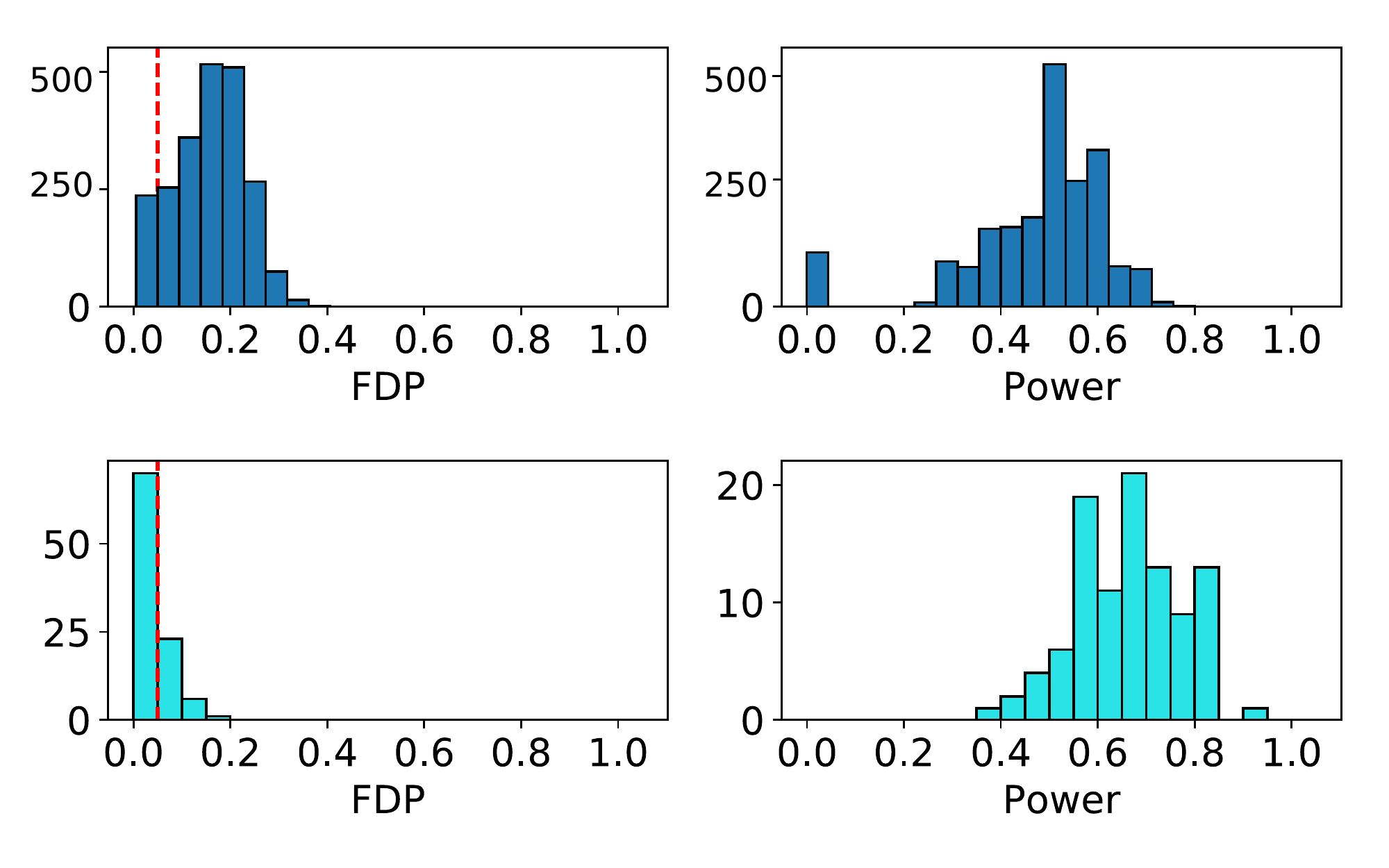}
  \caption {\textbf{Histogram of FDP and power for 2500 runs of KO (blue, top row) vs.
      100 runs of AKO with \( \mathbf{B=25} \) (teal, bottom row) under \underline{the same
        simulation}}. Simulation parameter:
    \( \text{SNR}=3.0, \rho=0.5, \text{ sparsity}=0.06\). FDR is controlled at
    level \( \alpha=0.05 \).}
  \label{fig:histo-fdp-power}
\end{figure}

\textbf{Inference Results on Different Simulation Settings.}
To observe how each algorithm performs under various scenarii, we vary each of
the three simulation parameters while keeping the others unchanged at default
value.
The result is shown in Figure~\ref{fig:fdr-power}.
Compared with KO, AKO improves statistical power while still controlling the
FDR.
Noticeably, in the case of very high correlation between nearby variables
(\( \rho > 0.7 \)), KO suffers from a drop in average power.
The loss also occurs, but is less severe for AKO.
Moreover, compared with simultaneous knockoff (KO-GZ), AKO gets better control
for FDR and a higher average power in the extreme correlation (high \( \rho \))
case.
Knockoff statistics aggregation (KO-HH), contrarily, is spurious: it detects
numerous truly significant variables with high average statistical power, but
at a cost of failure in FDR control, especially when the correlation parameter
$\rho$ gets bigger than~$0.6$.
Debiased Lasso (DL-BH) and KO-EK control FDR well in all scenarii, but are the
two most conservative procedures. %% out of the six.

\begin{figure}[h]
  \centering
  \includegraphics[width=0.7\textwidth]{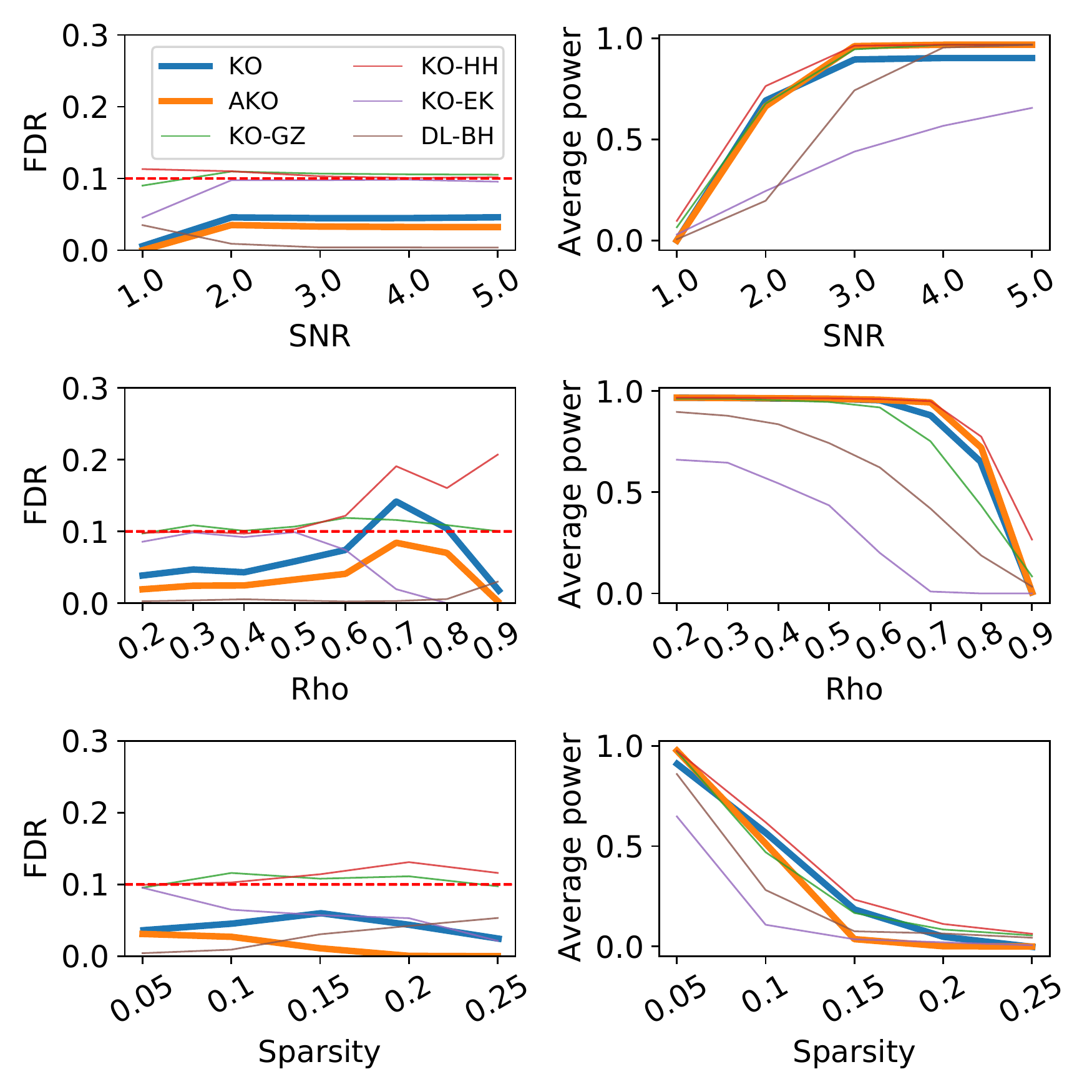}
  \caption {\textbf{FDR (left) and average power (right) of several methods for
      100 runs with varying simulation parameters}.
    For each varying parameter, we keep the other ones at default value:
    \( \text{SNR}=3.0, \rho=0.5, \text{ sparsity}=0.06 \).
    FDR is controlled at level \( \alpha=0.1 \).
    The benchmarked methods are: aggregation of multiple knockoffs (AKO -- ours); vanilla
    knockoff (KO);
    simultaneous knockoff by \cite{gimenez_improving_2019} (KO-GZ) and 
    by \cite{emery_controlling_2019} (KO-EK);
    knockoff statistics aggregation (KO-HH); debiased-Lasso (DL-BH).}
  \label{fig:fdr-power}
\end{figure}

\textbf{Choice of \( \mathbf{B} \) and \( \bm\gamma \) for AKO.}
Figure~\ref{fig:b-gamma-varying} shows an experiment when varying \(\gamma\) and
\( B \).
FDR and power are averaged across 30 simulations of fixed parameters: SNR=3.0,
\(\rho=0.7\), sparsity=0.06.
Notably, it seems that there is no further gain in statistical power when
\( B > 25 \).
Similarly, the power is essentially equal for \(\gamma\) values greater than 0.1
when \( B \geq 25 \).
Based on the results of this experiment we set the default value of
\( B = 25, \gamma=0.3 \).

\begin{figure}[t]
  \centering
  \includegraphics[width=0.5\textwidth]{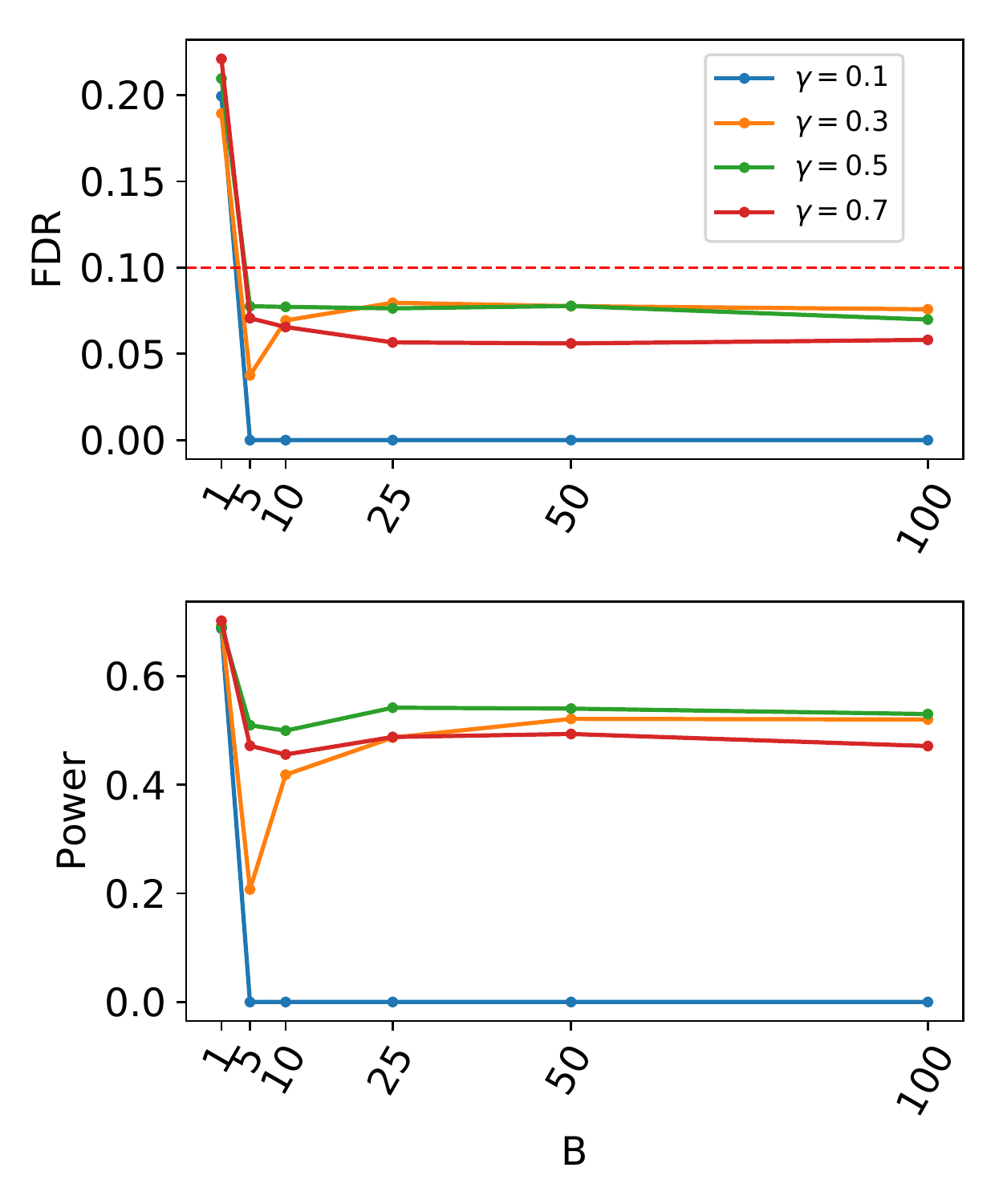}
  \caption {\textbf{FDR and average power for 30 simulations of fixed parameters:
    SNR=3.0, \(\mathbf{\bm\rho=0.7}\), sparsity=0.06.} %
    There is virtually no gain in statistical power when \( B > 25 \) and when
    \( \gamma \geq 0.1 \).}
  \label{fig:b-gamma-varying}
\end{figure}

\subsection{GWAS on Flowering Phenotype of \textit{Arabidopsis
  thaliana}}
\label{sec:experiments:GWAS}

To test AKO on real datasets, we first perform a genome-wide association study (GWAS) 
on genomic data.
The aim is to detect association of each of 174 candidate genes with a
phenotype \textbf{FT\_GH} that describes flowering time of \textit{Arabidopsis
  thaliana}, first done by \cite{atwell_genome_2010}.
Preprocessing is done similarly to \cite{azencott_efficient_2013}: 166 data
samples of 9938 binary SNPs located within a \( \pm20- \)kilobase window of 174
candidate genes that have been selected in previous publications as most likely
to be involved in flowering time traits.
Furthermore, we apply the same dimension reduction by hierarchical
clustering as \cite{slim_kernelpsi_2019} to make the final design matrix of size
$n=166$ samples \( \times \) $p=1560$ features.
We list the detected genes from each method in Table \ref{table:detected-genes}.

\begin{table}[h]
  \caption{\textbf{List of detected genes associated with phenotype
      FT\_GH}. Empty line (---) signifies no detection. Detected genes are
    listed in well-known studies dated up to 20 years ago.}
\label{table:detected-genes}
\vskip 0.15in
\begin{center}
\begin{small}
\begin{sc}
\begin{tabular}{ll}
\toprule
\textbf{Method}  & \textbf{Detected Genes} \\
\midrule
AKO    & AT2G21070, AT4G02780, AT5G47640 \\
KO     & AT2G21070 \\
KO-GZ  & AT2G21070 \\
DL-BH   & --- \\
\bottomrule
\end{tabular}
\end{sc}
\end{small}
\end{center}
\vskip -0.1in
\end{table}

The three methods that rely on sampling knockoff variables detect AT2G21070.
This gene, which is responsible for the mutant FIONA1, is listed by
\cite{kim_fiona1_2008} to be vital for regulating period length in the
\textit{Arabidopsis} circadian clock.
FIONA1 also appears to be involved in photoperiod-dependent flowering and in
daylength-dependent seedling growth.
In particular, the time for opening of the first flower for FIONA1 mutants
are shorter than the ones without under both long and short-day conditions.
In addition to FIONA1 mutant, AKO also detects AT4G02780 and AT5G47640.
It can be found in studies dating back to the 90s
\cite{silverstone_arabidopsis_1999} that AT4G02780 encodes a mutation for late
flowering.
Meanwhile, AT5G47640 mutant delay flowering in long-day but not in short-day
experiments \cite{cai_putative_2007}.

\subsection{Functional Magnetic Resonance Imaging (fMRI) analysis on Human
  Connectome Project Dataset}
\label{sec:experiments:HCP900}

Human Connectome Project (HCP900) is a collection of neuroimaging and behavioral
data on 900 healthy young adults, aged 22--35.
Participants were asked to perform different tasks inside an MRI scanner while
blood oxygenation level dependent (BOLD) signals of the brain were recorded.
The analysis investigates what brain regions are predictive of the subtle
variations of cognitive activity across participants, conditional to other brain
regions.
Similar to genomics data, the setting is high-dimensional with \(n = 1556 \)
samples acquired and 156437 brain voxels.
A voxel clustering step that reduces data dimension to \(p = 1000\) clusters is
done to make the problem tractable.

When decoding brain signals on HCP subjects performing a foot motion experiment
(Figure~\ref{fig:hcp-emotional}, left), AKO recovers an anatomically correct
anti-symmetric solution, in the motor cortex and the cerebellum, together with a
region in a secondary sensory cortex.
KO only detects a subset of those.
Moreover, across seven such tasks, the results obtained independently from
DL-BH are much more similar to AKO than to KO, as measured with Jaccard index
of the resulting maps (Figure~\ref{fig:hcp-emotional}, right).
The maps for the seven tasks are represented in the Appendix.
Note that the sign of the effect for significant regions is readily obtained
from the regression coefficients, with a voting step for bootstrap-based
procedures.
\begin{figure}[h]
  \centering
    \includegraphics[width=0.35\textwidth]{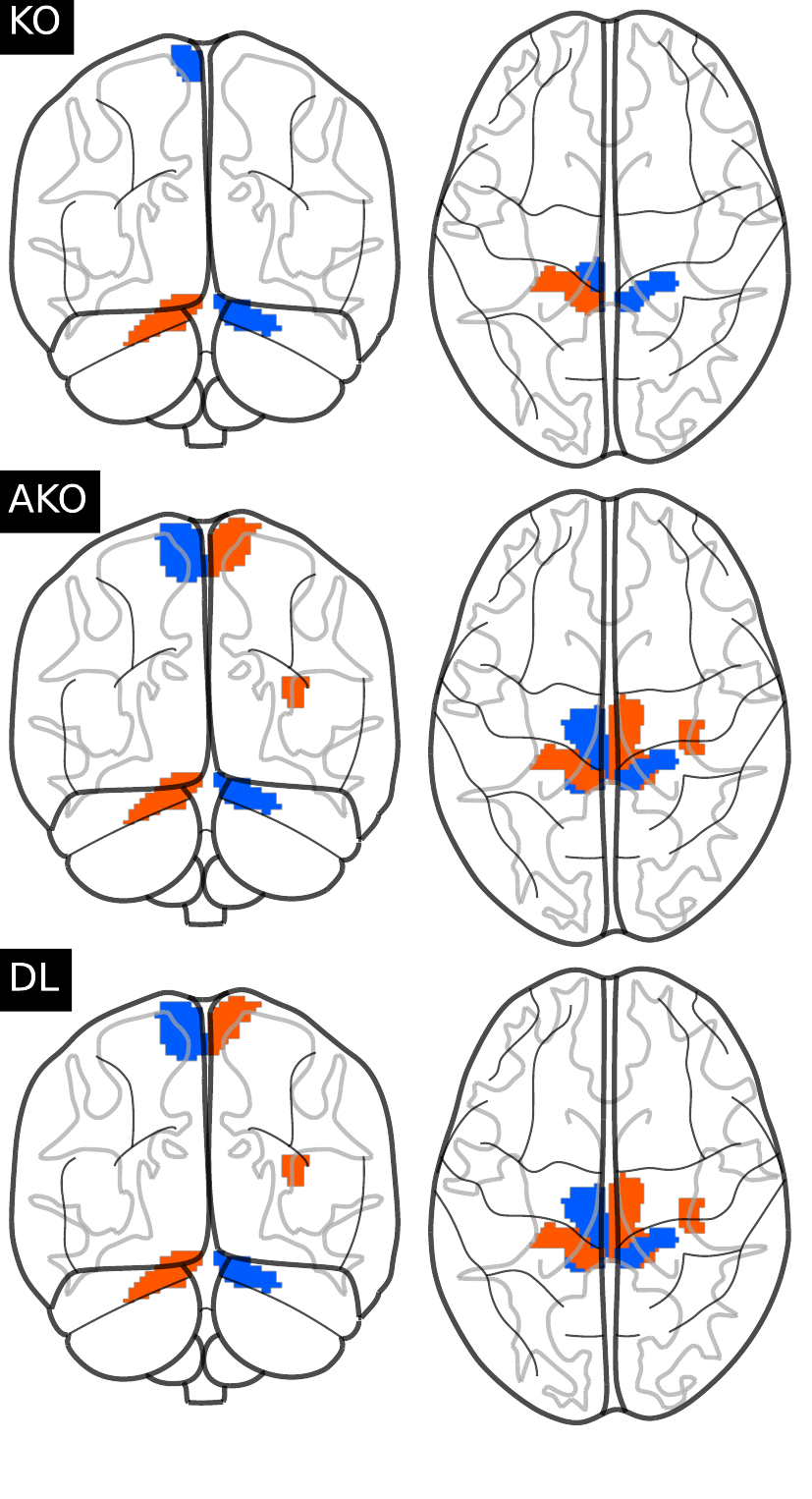}
    \includegraphics[width=0.32\textwidth]{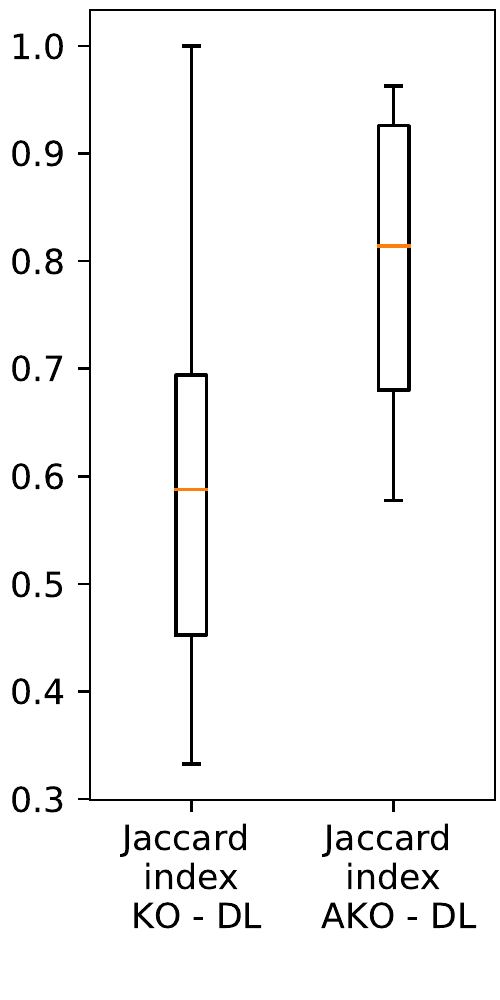}
  \caption{\textbf{Detection of significant brain regions for HCP data (900
      subjects).} (left) Selected regions in a left or right foot movement
    task. \textcolor{orange}{\textbf{Orange}}: brain areas with positive sign
    activation. \textcolor{blue}{\textbf{Blue}}: brain areas with negative sign
    activation. Here the AKO solution recovers an anatomically correct pattern,
    part of which is missed by KO. (right) Jaccard index measuring the Jaccard
    similarity between the KO/AKO solutions on the one hand, and the DL solution
    on the other hand, over 7 tasks: AKO is significantly more consistent with
    the DL-BH solution than KO.}
  \label{fig:hcp-emotional}
\end{figure}

%%%%%%%%%%%%%%%%%%%%%%%%%%%%%%%%%%%%%%%%%%%%%%%%%%%%%%%%%%%%%%%%%%%%%%%%%%%%%%%%

\section{Discussion}

In this work, we introduce a p-value to measure knockoff importance and
design a knockoffs bootstrapping scheme that leverages this quantity.
With this we are able to tame the instability inherent to the original knockoff
procedure.
Analysis shows that aggregation of multiple knockoffs retains theoretical
guarantees for FDR control.
However, \textit{i)} the original argument of \cite{barber_controlling_2015} no
longer holds (see Appendix); \textit{ii)} a factor $\kappa$ on the FDR control
is lost; this calls for tighter FDR bounds in the future, since we always
observe empirically that the FDR is controlled without the factor~$\kappa$.
Moreover, both numerical and realistic experiments show that performing
aggregation results in an increase in statistical power and also more
consistent results with respect to alternative inference methods.

The quantile aggregation procedure from \cite{meinshausen_p-values_2009} used
here is actually conservative: as one can see in Figure~\ref{fig:fdr-power}, the
control of FDR is actually stricter than without the
aggregation step.
Nevertheless, as often with aggregation-based approaches, the gain in accuracy
brought by the reduction of estimator variance ultimately brings more power.
%
% An interesting open question is whether less conservative aggregation schemes
% can be used instead.

%\medbreak

We would like to address here two potential concerns about FDR control for AKO+BH. 
The first one is when the $\{W_j \}_{j \in \cS^c}$ are not independent, hence
violating Assumption \ref{assumption:ko-stat-dist}.
In the absence of a proof of Theorem~\ref{thm:fdr-control} that would hold
under a general dependency, we first note that several schemes for knockoff
construction (for instance, the one of \cite{candes_panning_2018}) imply the
independence of $(\bx_i - \tilde{\bx}_i)_{i \in [p]}$, as well as their pseudo
inverse.
These observations do not establish the independence of $W_j$.  Yet,
intuitively, the Lasso coefficient of one variable should be much more
associated with its knockoff version than with other variables, so it should
not be much affected by these other variables, making the Lasso-coefficient
differences weakly correlated if not independent.
Moreover, in the proof of Lemma~\ref{lemma:pval} and
Theorem~\ref{thm:fdr-control}, Assumption \ref{assumption:ko-stat-dist} is only
used for applying Bernstein's inequality,
%%% applying Bernstein's inequality to the (1_{U_k \leq u})_k, where W_k=F_0^{-1}(U_k),
and several dependent versions of Bernstein's inequality have been proved
\cite[among others]{Sam:2000,Mer_Pel_Rio:2009,Han_Ste:2017}.  Similarly, the
proof of Eq.~\eqref{eq.pval.asympt} only uses Assumption
\ref{assumption:ko-stat-dist} for applying the strong law of large numbers, a
result which holds true for various kinds of dependent variables (for
instance, \cite{Abd:2018}, and references therein).
Therefore we conjecture that independence in Assumption
\ref{assumption:ko-stat-dist} can be relaxed into some mixing condition.
Overall, given that the unstability of KO with respect to the KO randomness is
an important drawback (see Figure~\ref{fig:histo-fdp-power}), we consider
Assumption \ref{assumption:ko-stat-dist} as a reasonable price price to pay for
correcting it, given that we expect to relax it in future works.
\\
The second potential concern is that Theorem~\ref{thm:fdr-control} is for AKO
with $\widehat{k}$ computed from the BY procedure, while BH step-up may not
control the FDR when the aggregated p-values $(\bar{\pi}_j)_{j \in [p]}$ are
not independent.
We find empirically that the $(\bar{\pi}_j)_{j \in [p]}$ do not exhibit
spurious Spearman correlation (Figure~\ref{fig:spearman} in Appendix) under a
setting where the $W_j$ satisfy a mixing condition.
This is a mild assumption that should be satisfied, especially when each
feature $X_j$ only depends on its "neighbors" (as typically observed on
neuroimaging and genomics data).
It is actually likely that the aggregation step contributes to reducing the
statistical dependencies between the $(\bar{\pi}_j)_{j \in [p]}$.
Eventually, it should be noted that BH can be replaced by BY
\cite{benjamini_control_2001} in case of doubt.
\\
To conclude on these two potential concerns, let us emphasize that the FDR of
AKO+BH with $B>1$ is always below $\alpha$ (up to error bars) in \emph{all} the
experiments we did, including preliminary experiments not shown in this
article, which makes us confident when applying AKO+BH on real data such as the
ones of Sections~\ref{sec:experiments:GWAS}--\ref{sec:experiments:HCP900}.

%%\medbreak

A practical question of interest is to handle the cases where $n\ll p$, that
is, the number of features overwhelms the number of samples. Note that in our
experiments, we had to resort to a clustering scheme of the brain data and to
select some genes.
A possible extension is to couple this step with the inference framework, in
order to take into account that for instance the clustering used is not given
but \emph{estimated} from the data, hence with some level of uncertainty.

The proposed approach introduces two parameters: the number $B$ of bootstrap
replications and the $\gamma$ parameter for quantile aggregation.
The choice of $B$ is simply driven by a compromise between accuracy (the larger
$B$, the better) and computation power, but we consider that much of the
benefit of AKO is obtained for $B \approx 25$.
Regarding $\gamma$, adaptive solutions have been proposed
\cite{meinshausen_p-values_2009}, but we find that choosing a fixed quantile
(0.3) yields a good behavior, with little variance and a good sensitivity.

\section*{Acknowledgements}

The authors thank anonymous reviewers for their helpful comments and
suggestions. We are grateful for the help of Lotfi Slim and Chlo\'e-Agathe
Azencott on the \textit{Arabidopsis thaliana} dataset, and the people of Human
Connectome Project on HCP900 dataset.

This research is supported under funding of French ANR project FastBig
(ANR-17-CE23-0011), the KARAIB AI chair and Labex DigiCosme
(ANR-11-LABEX-0045-DIGICOSME).

\bibliography{bibliography}
\bibliographystyle{alphaabbr}

\newpage
%{\center {\Large\bf Appendix}}
\section*{Appendix}
\appendix

%%%%%%%%%%%%%%%%%%%%%%%%%%%%%%%%%%%%%%%%%%%%%%%%%%%%%%%%%%%%%%%%%%%%%%%%%%%%%%%%

% for equation symbol in appendix section
\renewcommand{\theequation}{\thesection.\arabic{equation}}
\renewcommand{\thelemma}{\thesection.\arabic{lemma}}
\renewcommand{\theremark}{\thesection.\arabic{remark}}
\renewcommand{\theassumption}{\thesection.\arabic{assumption}}
\renewcommand{\theresult}{\thesection.\arabic{result}}
\renewcommand{\thefigure}{\thesection.\arabic{figure}}
%%%%%%%%%%%%%%%%%%%%%%%%%%%%%%%%%%%%%%%%%%%%%%%%%%%%%%%%%%%%%%%%%%%%%%%%%%%%%%%%

The Appendix is organized as follows.
First, the main theoretical results of the article are proved:
\begin{itemize}
\item Proof of Proposition~\ref{prop:equivalence}: AKO+BH with $B=1$ and
  $\gamma=1$ is equivalent to vanilla KO.
\item Proof of Lemma~\ref{lemma:ko-distinct}: for Lasso-coefficient differences,
  the non-zero $W_j$ are distinct.
\item Proof that the $\pi_j$ are \emph{asymptotically} valid p-values (without
  any multiplicative correction): Lemma~\ref{lemma:pval:asympt}.
\item Statement and proof of a new general result about FDR control with
  quantile-aggregated p-values: Lemma~\ref{le.FDR-aggregation}.
\item Proof of Theorem~\ref{thm:fdr-control}.
\end{itemize}
Second, the results of some additional experiments are reported: 
\begin{itemize}
\item Additional experiments to show that the KO-GZ alternative aggregation
  procedure by \cite{gimenez_improving_2019} has decreasing power when the
  number $\kappa$ of knockoff vectors $\tilde{\bx}$ considered simultaneously
  increases (we compare $\kappa=2$ with $\kappa=3$).  We show empirically that
  this is not the case for AKO with respect to~$B$.
\item Empirical evidence for the near independence of p-values \( \pi_j \).
\item Additional figures for HCP 900 experiments.
\end{itemize}

\section{Detailed Proofs}

\subsection{Proof of Proposition~\ref{prop:equivalence}}
\label{ssec:proof-prop-equiv}
We begin by noticing that the function \( f: \bbR^+ \rightarrow \bbZ^+ \),
\( f(x) = \dfrac{\#\{k: W_k \leq -x \}}{p} \) is decreasing in \( x \).
This means the first step of both FDR control step-up procedures, that involves
ordering the intermediate p-values ascendingly, is the same as arranging the
knockoff statistic in descending order:
\( W_{(1)} \geq W_{(2)} \geq \cdots \geq W_{(p)}\).
Therefore from Eq.~\eqref{eq:pval-threshold} and the definition of \( \pi_j \) we
have:
\[
  \widehat{k} = \max \left\{ k: \dfrac{1 + \#\{i: W_{(i)} \leq -W_{(k)} \}}{p} \leq
    \dfrac{k\alpha}{p} \right\}
\]
(note that we can exclude all the \( \pi_{(k)} = 1 \) due to the fact that
\( \forall \ k\in [p], \alpha \in (0, 1): k\alpha/p < 1 \)).

This can be written as:
\[
  \widehat{k} = \max \left\{ k: \dfrac{1 + \#\{i: W_{(i)} \leq -W_{(k)} \}}{\#\{i:
      W_{(i)} \geq W_{(k)} \}} \leq \alpha \right\},
\]
since \( \#\{i: W_{(i)} \geq W_{(k)} \} = k \) because
\( \{W_{(j)}\}_{j \in [p]} \) is ordered descendingly and because of the
assumption that non-zero LCD statistics are distinct.
Furthermore, finding the maximum index \( k \) of the descending ordered
sequence is equivalent to finding the minimum value in that sequence, or
\begin{equation*}\label{eq:ako-threshold-1}
  \widehat{k} = \min \left\{ W_{(k)} > 0: \dfrac{1 + \#\{i: W_{(i)} \leq -W_{(k)} \}}{\#\{i:
      W_{(i)} \geq W_{(k)} \}} \leq \alpha \right\},
\end{equation*}
since all \( W_{(j)} \leq 0 \) (corresponding with \( \pi_{(k)} = 1 \)) have
been excluded.
Without loss of generality, we can write:
\begin{equation*}\label{eq:ako-threshold-2}
  \widehat{t}_+ = \min \left\{ t > 0: \dfrac{1 + \#\{i: W_{i} \leq -t \}}{\#\{i:
      W_{i} \geq t \}} \leq \alpha \right\} \, . 
\end{equation*}
This threshold \( \widehat{t}_+ \) is exactly the same as the definition of
threshold \(\tau_+ \) in Eq.~\eqref{eq:ko-threshold} from the original KO procedure.
\qed

\subsection{Proof of Lemma~\ref{lemma:ko-distinct}}
\label{ko-distinct}

\textbf{Setting.}
Let \( \bX \in \R^{n \times q} \),
\( \bm\beta^* \in \R^q, \lambda > 0, \sigma>0 \) be fixed.
Define
\[
  \by = \bX \bm\beta^* + \bm\epsilon
\] 
\[
  \forall \bm\beta \in \R^q, \qquad L(\bm\beta) := \lVert \by - \bX \bm\beta
  \rVert^2_2 + \lambda \lVert \bm\beta \rVert_1 \, .
\]
with \( \bm\epsilon \sim \cN(\mathbf{0}, \sigma\bI_n) \) the Gaussian noise and
\( \lVert \cdot \rVert_p \) the \( L_p \) norm.

\textbf{Classical Optimization Properties.}
Since $L$ is convex, non-negative, and tends to $+\infty$ at infinity, its
minimum over $\R^q$ exists and is attained (although may not be unique).
Since $L$ is convex, its minima are characterized by a first-order condition:
\[
  \notag
  \bhl \in \argmin_{\bm\beta \in \R^q} \set{ L(\bm\beta) }
  \Leftrightarrow \qquad 0 \in \partial L (\bm\beta)
\]
which is equivalent to 
\begin{gather}
\left\{
\begin{split}
  \exists \ \zh\in[-1,1]^q: \bX^{\top} \bX \bhl = \bX^{\top} \by -
  \frac{\lambda}{2} \bm\zh
  \\
  \forall j \text{ s.t. } (\bhl)_j \neq 0, \, \zh_j = \sign( (\bhl)_j )
\end{split}
\right. 
\label{eq.1st-order-condition}
\end{gather}

As shown by \cite[Section~4.5.1]{giraud_introduction_2014} for instance, the
fitted value $\fhl \in \R^n$ is uniquely defined:
\[ 
  \exists! \ \fhl \in \R^n 
\quad \text{such that} \quad
\forall 
\bm\bhl \in \argmin_{\bm\beta \in \R^q} \set{ L(\bm\beta) } 
\, , 
\qquad 
\fhl = \bX \bm\bhl
\, . 
\]
As a consequence, the equicorrelation set 
\[ 
\Jhl := \set{j \in \set{1,\ldots , q}:
\abs{ \bx_j^{\top} (\by - \bX \bm\bhl) } = \lambda/ 2 }
\] 
is uniquely defined. 
We also have,
\begin{equation}
\label{eq.Jhl-active}
\forall \ \bhl \in \argmin_{\bm\beta \in \R^q} \set{ L(\bm\beta) } 
\, , \qquad 
\set{j: (\bhl)_j \neq 0 } \subset \Jhl
\end{equation}
(but these two sets are not necessarily equal, and the former set may not be
uniquely defined).

Note that for every set $J \subset \set{1, \ldots , q}$ such that
$\forall j \notin J$, $(\bhl)_j = 0$, we have $\bX \bhl = \bX_J (\bhl)_J$ so that
$(\bX^{\top} \bX \bhl)_J = \bX^{\top}_J \bX_J (\bhl)_J$.
As a consequence, taking $J=\Jhl$, by eq. \refp{eq.1st-order-condition}
and~\eqref{eq.Jhl-active}, any minimizer $\bhl$ of $L$ over $\R^q$ satisfies
\begin{equation} 
\label{eq.1st-order.Jhl}
\bX^{\top}_{\Jhl} \bX_{\Jhl} (\bhl)_{\Jhl} = \bX^{\top}_{\Jhl} \by - \frac{\lambda}{2} \zh_{\Jhl} 
\end{equation}
for some $\zh_{\Jhl} \in \set{-1,1}^{\Jhl}$. 

If the matrix $\bX_{\Jhl}^{\top} \bX_{\Jhl}$ is non-singular (that is, if
$\bX_{\Jhl}$ is of rank $|\Jhl|$), then the argmin of $L$ is unique
\cite[Section~4.5.1]{giraud_introduction_2014}.

\begin{result}\label{res:le.W-distinct}
  For every $\bm\alpha \in \R^q \backslash \{0\}$, the event
  \begin{gather}\label{eq.le.W-distinct}
    \rank(\bX_{\Jhl}) = |\Jhl| \, , \quad \bm\alpha^{\top} \bhl = 0 \quad
    \text{and} \quad \exists \ j \in \set{1, \ldots , q} \, , \quad \alpha_j
    (\bhl)_j \neq 0 \, ,
  \end{gather}
  where $\{ \bhl \} = \argmin_{\bm\beta \in \R^q} \set{ L(\bm\beta) }$ is
  well-defined by the first property, has probability zero.
\end{result}
\begin{proof}
  Let $\Omega$ be the event defined by Eq.~\eqref{eq.le.W-distinct}.
  If $\Omega$ holds true, then there exists some $J \subset \set{1, \ldots, q}$
  and some $\zh \in \{-1,1\}^q$ such that $\rank(X_J)=|J|$, $(\bhl)_{J^c} = 0$,
  and
  \[ 
    \bX^{\top}_{J} \bX_{J} (\bhl)_{J} = \bX^{\top}_{J} \by - \frac{\lambda}{2} \zh_{J} 
    \, .
  \]
  Indeed, this is a consequence of Eq.~\eqref{eq.Jhl-active} and
  \eqref{eq.1st-order.Jhl}, by taking $J=\Jhl$ and $\bz$ such that
  $\bz_{\Jhl} = \sign \paren{ (\bhl)_{\Jhl} }$.
  Therefore, using that $\bX^{\top}_{J} \bX_{J}$ is non-singular, we get
  \begin{align*} 
    (\bhl)_{J} &= M(J) \bm\epsilon + v(J,\bz)
    \\ 
    \text{where} \qquad 
    M(J) &:= ( \bX^{\top}_{J} \bX_{J} )^{-1} \bX^{\top}_{J} 
    \\ 
    \text{and} \qquad 
    v(J,\bz) &:= ( \bX^{\top}_{J} \bX_{J} )^{-1} \bX^{\top}_{J} \bX \bm\beta^*
               - ( \bX^{\top}_{J} \bX_{J} )^{-1} \frac{\lambda}{2} \zh_{J} 
               \, , 
  \end{align*}
  hence 
  \[ \bm\alpha^{\top} \bhl = \bm\alpha_J^{\top} M(J) \bm\epsilon +
    \bm\alpha^{\top} v(J,\bz)
  \]
  follows a normal distribution with variance
  $\sigma^2 \bm\alpha_J^{\top} M(J) M(J)^{\top} \bm\alpha_J = \sigma^2 \norm{M(J)^{\top}
    \bm\alpha_J}^2$.
  Now, on $\Omega$, we also have the existence of some $j$ such that
  $\alpha_j (\bhl)_j \neq 0$.
  Since $(\bhl)_{J^c}=0$, we must have $j \in J$, which shows that
  $\bm\alpha_J \neq 0$.

  Overall, we have proved that
  \begin{align*} 
    \Omega &\subset \bigcup_{J \in \mathcal{J}, \bz \in \{-1,1\}^q } \Omega_{J,z} 
    \\ 
    \text{where} \qquad 
    \mathcal{J} &:= \set{j \in \set{1, \ldots, q}:
                  \rank(X_J)=|J| \text{ and } \alpha_J \neq 0}
    \\ 
    \text{and} \qquad 
    \Omega_{J,\bz} &:= \set{ \bm\alpha_J^{\top} M(J) \bm\epsilon + \alpha^{\top} v(J,\bz) = 0 } 
                     \, . 
  \end{align*}
  For every $J \in \mathcal{J}$, $M(J)^{\top} \bm\alpha_J \neq 0$ since
  $\alpha_J \neq 0$ and $M(J)$ is of rank $|J|$.
  As a consequence, for every $J \in \mathcal{J}$ and $\bz \in \set{-1,1}^p$,
  $\P( \Omega_{J,\bz} ) $ is the probability that a Gaussian variable with non-zero
  variance is equal to zero, so it is equal to zero.
  We deduce that 
  \[ 
    \P( \Omega ) \leq \sum_{J \in \mathcal{J}, \bz \in \{-1,1\}^q } \P (\Omega_{J,z})  =0
  \]
  since the sets $\mathcal{J}$ and $\set{-1,1}^q$ are finite.  
\end{proof}

Applying Result \ref{res:le.W-distinct} to the case where $\bX$ concatenates the
original $p$ covariates and their knockoff counterparts (hence $q=2p$), we get
that, apart from the event where $\bX_{\Jhl}$ is not full rank, for every
$j \in \set{1, \ldots, p}$, $W_j$ takes any fixed non-zero value with
probability zero (with $\alpha_j = \pm 1$, $\alpha_{j+p}=\pm 1$, $\alpha_k=0$
otherwise).

Similarly, the above lemma shows that for every
$j \neq j' \in \set{1, \ldots, p}$:
\[
  \P( \bX_{\Jhl} \text{ is full-rank and } \exists j \neq j', W_j = W_{j'}, W_j
  \neq 0, W_{j'} \neq 0 ) = 0.
\]
As a consequence, with probability 1, all the non-zero $W_j$ are distinct if
$X_{\Jhl}$ is full-rank. 
\qed 

\begin{remark}
  The proof of Result~\ref{res:le.W-distinct} is also valid for other noise
  distributions: it only assumes that the support of the distribution of
  $\bm\epsilon$ is not included into any hyperplane of $\R^n$.
\end{remark}

\subsection{Asymptotic Validity of Intermediate P-values}
We consider in this section an asymptotic regime where $p \to +\infty$. 

\begin{assumption}[Asymptotic regime $p \to +\infty$]
\label{hyp.asympt}
When $p$ grows to infinity, $n$, $\bX$, $\beta^*$, $\bm\epsilon$ and $\by$
all depend on $p$ implicitly.
We assume that for every integer $j \geq 1$, $\mathbbm{1}_{\beta_j^* = 0}$ does
not depend on $p$ (as soon as $p \geq j$), and that
\[ 
  \frac{\abs{\cS}}{p} 
  = \frac{\bigl\lvert \{ j \in [p] : \beta_j^* \neq 0  \} \bigr\rvert }{p} 
  \xrightarrow[p \rightarrow +\infty]{} 0 
  \, . 
\]
When making Assumption~\ref{assumption:ko-stat-dist}, we also assume that
$\bbP_0$ does not depend on~$p$.
\end{assumption}

\begin{lemma}
\label{lemma:pval:asympt}
If Assumptions~\ref{assumption:ko-stat-dist} and~\ref{hyp.asympt} hold true, then for
all \( j \geq 1 \) such that \( \beta^*_j = 0 \), the empirical p-value
\( \pi_j \) defined by Eq.~\eqref{eq:pval} is a valid p-value asymptotically,
that is,
\[
 \forall t \in [0,1] \, , \qquad \lim_{p \to +\infty} \bbP(\pi_j \leq t) \leq t
 \, . 
\]
\end{lemma}

Note that our proof of Theorem~\ref{thm:fdr-control} in
Section~\ref{sec.pr.ThmPpal} relies on the use of Lemma~\ref{lemma:pval} with
$t$ that can be of order $1/p$.
Therefore, Lemma~\ref{lemma:pval:asympt} above is not sufficient for our needs.
Nevertheless, it still provides a interesting insight about the $\pi_j$, and
justifies (asymptotically) their name, which is why we state and prove this
result here.

\begin{proof}
   By definition, \( \pi_j \leq 1 \) almost surely, so the result holds when
   \( t = 1 \).  Let us now focus on the case where \( t \in [0, 1) \).
   Let \( F_0 \) denote the c.d.f. of \( \bbP_0 \), the common distribution of
   the null statistics
   \(\discset{W_j}_{1 \leq j \leq p \,/\, \beta^*_j = 0} \), which exists by
   Assumption~\ref{assumption:ko-stat-dist}.
   Let \( j \geq 1 \) such that \( \beta^*_j = 0 \) be fixed, and assume that
   $p \geq j$ is large enough so that $\abs{\cS^c} \geq 2$.  Let
   $m = \lvert \cS^c \rvert - 1 \geq 1$ as in the proof of
   Lemma~\ref{lemma:pval}.
   Note that $m$ depends on $p$, and $m/p \to 1$ as $p \to +\infty$ by
   Assumption~\ref{hyp.asympt}, hence $m \to +\infty$ as $p \to +\infty$.

   By definition of \( \pi_j \), when \( W_j > 0 \) we have:
   \begin{align*}
     \pi_j &= \dfrac{1 + \#\{k \in [p]: W_k \leq -W_j\}}{p} \\
     \text{(since \( W_j > 0 > - W_j \))} &= \dfrac{1 + \#\{k \in \cS: W_k \leq -W_j\} +
             \#\{k \in \cS^c\setminus\{ j \}: W_k \leq -W_j\}}{p} \\
           &\geq \dfrac{\#\{ k \in \cS^c\setminus\{ j \}: W_k \leq -W_j \} }{p} 
        \\ &= \frac{\widehat{F}_{m} (-W_j) }{\alpha_p} 
             \numberthis \label{eq:lower-bound}
   \end{align*}
   where $\alpha_p \egaldef \dfrac{p}{m}$ and for all $u \in \R$,
   \[
     \widehat{F}_{m} (u) \egaldef \dfrac{\#\{ k \in \cS^c\setminus\{ j \}: W_k
       \leq u \} }{m}     
   \]
   is the empirical cdf of \( \{ W_k\}_{k \in \cS^c \setminus \{j\}} \).

   Now, since \( \{ W_k\}_{k \in \cS^c \setminus \{j\}} \) are \iid with
   distribution \( \bbP_0 \) by Assumption~\ref{assumption:ko-stat-dist} , the
   law of large numbers implies that, for all \( u \in \bbR \),
   \begin{align*}
     \widehat{F}_{m} (u)
     \xrightarrow[p \to +\infty]{\text{a.s.}} 
     F_0(u) 
     \, . 
   \end{align*}
   Since we assume $\lim_{p \to +\infty} \abs{\cS}/p = 0$,
   $\lim_{p\to+\infty} \alpha_p = 1$ and we get that for all \( u \in \bbR \),
   \begin{align*}
     \frac{1}{\alpha_p} \widehat{F}_{m} (u)
     \xrightarrow[p \to +\infty]{\text{a.s.}} 
     F_0(u) 
     \, . 
   \end{align*}
   Since $W_j$ is independent from \( \{ W_k\}_{k \in \cS^c \setminus \{j\}} \),
   this result also holds true \emph{conditionally to $W_j$}, with $u=-W_j$.
   Given that almost sure convergence implies convergence in distribution, we
   have: conditionally to $W_j$,
   \begin{equation}
     \label{pr.main-lemma:pval.cv-distrib}
     \frac{1}{\alpha_p} \widehat{F}_{m} (-W_j)
     \xrightarrow[p \to +\infty]{\text{(d)}} 
     F_0(-W_j) \stackrel{(d)}{=} F_0(W_j) 
   \end{equation}
   where the latter equality comes from the fact that $W_j$ has a symmetric
   distribution, as shown in Remark~\ref{remark:P0-symmetric}.

   So, when $W_j>0$, for every $t \in [0,1)$, 
   \begin{align}
     \notag 
     \limsup_{p \to +\infty} \bbP( \pi_j \leq t \,\vert\, W_j ) 
     &\leq \limsup_{p \to +\infty} \bbP \left( \frac{\widehat{F}_{m} (-W_j) }{\alpha_p} \leq t \,\Big\vert\, W_j \right) 
       \qquad \text{by Eq.~\eqref{eq:lower-bound}}
     \\
     &\leq \mathbbm{1}_{F_0(W_j) \leq t} 
       \label{eq:convergence2}
   \end{align}
   by Eq.~\eqref{pr.main-lemma:pval.cv-distrib} combined with the Portmanteau theorem. 

   Therefore, for every $t \in [0,1)$, 
   \begin{align*}
     \limsup_{p \rightarrow +\infty} \bbP(\alpha_p\pi_j \leq t)
     &= \limsup_{p \rightarrow +\infty} \Bigl\{ \bbP(\alpha_p\pi_j \leq t \
       \text{and} \ W_j > 0 )
       + \underbrace{\bbP(\alpha_p\pi_j \leq t \ \text{and} \ W_j \leq 0) }_{\substack{=0 \text{
       since } \alpha_p \geq 1 > t \\ \text{ and } \pi_j = 1 \text{ when } W_j \leq 0}} \Bigr\}
     \\
     &= \limsup_{p \rightarrow +\infty} \bbE[\bbP(\alpha_p\pi_j \leq t \mid
       W_j) \mathbbm{1}_{W_j > 0 }] 
       \\
     &\leq \bbE\left[ \limsup_{\abs{\cS^c} \rightarrow +\infty} \{ \bbP(\alpha_p\pi_j \leq t
       \mid W_j) \mathbbm{1}_{W_j > 0}) \} \right] 
       \\
     &\leq \bbP \bigl( F_0(W_j) \leq t \bigr)
       \qquad \text{by Eq.~\eqref{eq:convergence2}}
     \\ 
     &\leq t
     \, ,    
   \end{align*}
   which concludes the proof. 
 \end{proof}

\subsection{A General FDR Control with Quantile-aggregated P-values}
The proof of Theorem~\ref{thm:fdr-control} relies on an adaptation of results
proved by \cite[Theorems 3.1 and~3.3]{meinshausen_p-values_2009} about
aggregation of p-values.
The results of \cite{meinshausen_p-values_2009}, whose proof relies on the
proofs of \cite{benjamini_control_2001}, are stated for randomized p-values
obtained through sample splitting.
The following lemma shows that they actually apply to any family of p-values.
\begin{lemma}
\label{le.FDR-aggregation}
Let $(\pi_j^{(b)})_{1 \leq j \leq p , 1\leq b \leq B}$ be a family of random
variables with values in $[0,1]$.
Let $\gamma \in (0,1]$, $\bar{\alpha} \in [0,1]$ and $\mathcal{N} \subset [p]$
be fixed.
Let us define
\begin{align*}
  \forall j \in [p] \, , \quad 
  Q_j 
  &\egaldef \frac{p}{\gamma} q_{\gamma} \bigl( \{ \pi_j^{(b)} \,:\,  1 \leq b \leq B \} \bigr) 
    \quad \text{where } \quad q_{\gamma}(\cdot) \text{ is the $\gamma$-quantile function,} 
  \\
  \widehat{h} 
  &\egaldef \max \bigl\{ i \in [p] \, : \, Q_{(i)} \leq i \bar{\alpha} \bigr\} 
    \quad \text{where} \quad Q_{(1)} \leq \cdots \leq Q_{(p)} 
    \, , 
  \\ 
  \text{and} \quad 
  \widehat{S} 
  &\egaldef \bigl\{ j \in [p] \,:\, Q_j \leq Q_{(\widehat{h})} \bigr\} 
    \, . 
\end{align*} 
Then, 
\begin{align}
  \label{eq.le.FDR-aggregation.res}
  \bbE \left[ \frac{\lvert \widehat{S} \cap \mathcal{N} \rvert}{\lvert \widehat{S} \rvert \vee 1} \right] 
  \leq \sum_{j=1}^{p-1} \left( \frac{1}{j} - \frac{1}{j+1} \right) F(j) + \frac{F(p)}{p} 
  \\
  \notag 
  \text{where} \qquad \forall j \in [p] \, , \, \qquad 
  F(j) \egaldef \frac{1}{\gamma} \frac{1}{B} \sum_{b=1}^B \sum_{i \in \mathcal{N}} 
  \bbP \left( \pi_i^{(b)} \leq \frac{j \bar{\alpha} \gamma}{p} \right) 
  \, . 
\end{align}
As a consequence, if some $C \geq 0$ exists such that 
\begin{equation}
  \label{eq.le.FDR-aggregation.cond-simplified}
  \forall t \geq 0, \forall b \in [B] \, , \, \forall i \in \mathcal{N} \, , 
  \qquad 
  \bbP \left( \pi_i^{(b)} \leq t \right) \leq C t \, , 
\end{equation}
then we have 
\begin{equation}
  \label{eq.le.FDR-aggregation.res-simplified}
  \bbE \left[ \frac{\lvert \widehat{S} \cap \mathcal{N} \rvert}{\lvert \widehat{S} \rvert \vee 1} \right]  
  \leq  \frac{\lvert \mathcal{N} \rvert C}{p}  \left( \sum_{j=1}^{p} \frac{1}{j} \right) \bar{\alpha} 
\, . 
\end{equation}

\end{lemma}
Let us emphasize that Lemma~\ref{le.FDR-aggregation} can be useful in general,
well beyond knockoff aggregation.
To the best of our knowledge, Lemma~\ref{le.FDR-aggregation} is new.
In particular, the recent preprint by \cite{Rom_DiC:2019}, that studies
p-values aggregation procedures, focuses on FWER controlling procedures,
whereas Lemma~\ref{le.FDR-aggregation} provides an FDR control for a less
conservative procedure.
\begin{proof}
For every $i,j,k \in [p]$, let us define 
\[ 
  p_{i,j,k} = 
  \begin{cases}
    \bbP \Bigl( Q_i \in \bigl( (j-1) \bar{\alpha}  , j\bar{\alpha}  \bigr] \, , \, i \in \widehat{S} \text{ and } \lvert \widehat{S} \rvert = k \Bigr)
    \qquad &\text{if } j \geq 2
    \\ 
    \bbP \bigl( Q_i \in [0 , \bar{\alpha} ] \, , \, i \in \widehat{S} \text{ and } \lvert \widehat{S} \rvert = k \bigr)
    \qquad &\text{if } j = 1 
    \, .  
  \end{cases}
\]
Then, 
\begin{align*}
  \frac{\lvert \widehat{S} \cap \mathcal{N} \rvert}{\lvert \widehat{S} \rvert \vee 1}
  &= \sum_{k=1}^p \mathbbm{1}_{\lvert \widehat{S} \rvert =k} \frac{\sum_{i \in \mathcal{N}} \mathbbm{1}_{i \in \widehat{S}}}{k} 
  \\
  &= \sum_{i \in \mathcal{N}} \sum_{k=1}^p \frac{1}{k} \mathbbm{1}_{\lvert \widehat{S} \rvert =k \text{ and } i \in \widehat{S}}
  \\
  &= \sum_{i \in \mathcal{N}} \sum_{k=1}^p \frac{1}{k} \mathbbm{1}_{\lvert \widehat{S} \rvert =k \text{ and } i \in \widehat{S} \text{ and } 0 \leq Q_i \leq k \bar{\alpha} }
\end{align*}
since $i \in \widehat{S}$ and $\lvert \widehat{S} \rvert =k$ implies that $Q_i \leq Q_{(\widehat{h})} \leq \widehat{h} \bar{\alpha} = k \bar{\alpha} $. 
Taking an expectation and writing that 
\[ 
\mathbbm{1}_{0 \leq Q_i \leq k \bar{\alpha} } 
= \mathbbm{1}_{Q_i \in [0,\bar{\alpha} ]} +  \sum_{j = 2}^k \mathbbm{1}_{Q_i \in ((j-1)\bar{\alpha} ,j\bar{\alpha} ]}
\, , 
\]
we get ---following the computations of \cite[proof of
Theorems~3.3]{meinshausen_p-values_2009}, which themselves rely on the ones of
\cite{benjamini_control_2001}---,
\begin{gather*}
  \begin{align*}
    \bbE \left[ \frac{\lvert \widehat{S} \cap \mathcal{N} \rvert}{\lvert \widehat{S} \rvert \vee 1} \right]  
    &\leq 
      \sum_{i \in \mathcal{N}} \sum_{k=1}^p \frac{1}{k} \sum_{j=1}^k p_{i,j,k} 
      = \sum_{i \in \mathcal{N}} \sum_{j=1}^p \sum_{k=j}^p \frac{1}{k} p_{i,j,k} \\
    &\leq \sum_{i \in \mathcal{N}} \sum_{j=1}^p \sum_{k=j}^p \frac{1}{j} p_{i,j,k} 
      = \sum_{j=1}^p \frac{1}{j}  \underbrace{ \sum_{i \in \mathcal{N}} \sum_{k=j}^p p_{i,j,k} }_{= \overline{F}(j) - \overline{F}(j-1) \mathbbm{1}_{j \geq 2}}    
  \end{align*}
\\
\text{where} \qquad \forall j \in \{1, \ldots, p\} \, , \qquad 
\overline{F}(j) \egaldef  
\sum_{i \in \mathcal{N}} \sum_{j'=1}^j  \sum_{k=1}^p p_{i,j',k} 
\, . 
\end{gather*}
Since the above upper bound is equal to
\begin{gather}
\notag 
\overline{F}(1) + \sum_{j=2}^p \frac{1}{j} \bigl[ \, \overline{F}(j) - \overline{F}(j-1) \bigr]
= \sum_{j=1}^p \left( \frac{1}{j} - \frac{1}{j+1} \right) \overline{F}(j) 
+ \frac{ \overline{F}(p) }{p} \, , 
\\
\text{we get that} \qquad \qquad 
\bbE \left[ \frac{\lvert \widehat{S} \cap \mathcal{N} \rvert}{\lvert \widehat{S} \rvert \vee 1} \right]  
\leq \sum_{j=1}^p \left( \frac{1}{j} - \frac{1}{j+1} \right) \overline{F}(j) 
+ \frac{ \overline{F}(p) }{p}
\label{eq.pr.le.FDR-aggregation}
\, . 
\end{gather}
Notice also that 
\[ 
\overline{F}(j) 
= \sum_{i \in \mathcal{N}} \bbP( Q_i \leq j \bar{\alpha}  \text{ and } i \in \widehat{S} )
\leq \sum_{i \in \mathcal{N}} \bbP( Q_i \leq j \bar{\alpha}  )
\enspace . 
\]
Now, as done by \cite[proof of Theorems~3.1]{meinshausen_p-values_2009}, we
remark that $Q_i \leq j \bar{\alpha} $ is equivalent to
\[ 
\frac{1}{B} \abs{ \left\{ b \in [B] \,:\, \frac{p \pi_i^{(b)}}{\gamma} \leq j \bar{\alpha} \right\} }  
= \frac{1}{B} \sum_{b=1}^B \mathbbm{1}_{p \pi_i^{(b)} \leq j \bar{\alpha}  \gamma} \geq \gamma 
\]
so that 
\begin{align*}
\bbP( Q_i \leq j \bar{\alpha}  )
&= \bbP \left( \frac{1}{B} \sum_{b=1}^B \mathbbm{1}_{p \pi_i^{(b)} \leq j
                                   \bar{\alpha}  \gamma} \geq \gamma  \right) \\
&\leq \frac{1}{\gamma} \bbE \left[ \frac{1}{B} \sum_{b=1}^B \mathbbm{1}_{p \pi_i^{(b)} \leq j \bar{\alpha}  \gamma}  \right] 
\qquad \text{by Markov inequality} 
\\
&= \frac{1}{\gamma}  \frac{1}{B} \sum_{b=1}^B \bbP \left( p \pi_i^{(b)} \leq j \bar{\alpha}  \gamma \right) 
\, .  
\end{align*}
Therefore, 
\[ 
\overline{F}(j) 
\leq \sum_{i \in \mathcal{N}} \frac{1}{\gamma}  \frac{1}{B} \sum_{b=1}^B \bbP \left( p \pi_i^{(b)} \leq j \bar{\alpha}  \gamma \right) 
= F(j) 
\, , 
\]
so that Eq.~\eqref{eq.pr.le.FDR-aggregation} implies
Eq.~\eqref{eq.le.FDR-aggregation.res}.

If condition~\eqref{eq.le.FDR-aggregation.cond-simplified} holds true, 
then, for every $j \in [p]$, 
\[ 
F(j) 
\leq \frac{\lvert \mathcal{N} \rvert C }{\gamma} \frac{j \bar{\alpha} \gamma}{p}
= \frac{\lvert \mathcal{N} \rvert C \bar{\alpha}}{p} j 
\, , 
\]
hence Eq.~\eqref{eq.le.FDR-aggregation.res} shows that 
\begin{align*}
\bbE \left[ \frac{\lvert \widehat{S} \cap \mathcal{N} \rvert}{\lvert \widehat{S} \rvert \vee 1} \right] 
\leq \sum_{j=1}^{p-1} \frac{F(j)}{j (j+1)} + \frac{F(p)}{p} 
\leq \frac{\lvert \mathcal{N} \rvert C \bar{\alpha}}{p} 
\sum_{j=1}^{p-1} \frac{1}{j+1} + \frac{\lvert \mathcal{N} \rvert C \bar{\alpha}}{p} 
= \frac{\lvert \mathcal{N} \rvert C \bar{\alpha}}{p}  \left( \sum_{j=1}^{p} \frac{1}{j} \right)
\, ,
\end{align*}
which is the desired result. 
\end{proof}

\subsection{Proof of Theorem~\ref{thm:fdr-control}}
\label{sec.pr.ThmPpal}
We can now prove Theorem~\ref{thm:fdr-control}. 
We apply Lemma~\ref{le.FDR-aggregation} 
with $\bar{\alpha} = \beta(p) \alpha$, 
$\mathcal{N} = \cS^c$, 
so that $\widehat{S} = \widehat{\mathcal{S}}_{AKO+BY}$. 
Since the $\pi_j^{(b)}$, $b=1, \ldots, B$, have the same distribution as $\pi_j$, 
by Lemma~\ref{lemma:pval}, condition~\eqref{eq.le.FDR-aggregation.cond-simplified} holds true with $C=\kappa p / \lvert \cS^c \rvert$, 
and Eq.~\eqref{eq.le.FDR-aggregation.res-simplified} 
yields the desired result. 
\qed

Note that an FDR control for AKO such as Theorem~\ref{thm:fdr-control} cannot
be obtained straightforwardly from the arguments of
\cite{barber_controlling_2015} and \cite{candes_panning_2018}.
One key reason for this is that their proof relies on a reordering of the
features according to the values of $(\abs{W_j})_{j \in [p]}$
\cite[Section~5.2]{barber_controlling_2015}, such a reordering being permitted
since the signs of the $W_j$ are iid coin flips \emph{conditionally to} the
$(\abs{W_j})_{j \in [p]}$ \cite[Lemma~2]{candes_panning_2018}.
In the case of AKO, we must handle the $(W_j^{(b)})_{j \in [p]}$
\emph{simultaneously for all $b \in [B]$}, and conditioning with respect to
$(\lvert W_j^{(b)} \rvert)_{j \in [p] , b \in [B]}$ may reveal some information
about the signs of the $(W_j^{(b)})_{j \in [p]}$ as soon as $B > 1$.
At least, it does not seem obvious to us that the key result of
\cite[Lemma~2]{candes_panning_2018} can be proved \emph{conditionally to} the
$(\lvert W_j^{(b)} \rvert )_{j \in [p] , b \in [B]}$ when $B>1$, so that the
proof strategy of \cite{candes_panning_2018} breaks down in the case of AKO
with $B>1$.

%%%%%%%%%%%%%%%%%%%%%%%%%%%%%%%%%%%%%%%%%%%%%%%%%%%%%%%%%%%%%%%%%%%%%%%%%%%%%%%%

\section{Additional Experimental Results}
\subsection{Demonstration of Aggregated Multiple Knockoff vs. Simultaneous
  Knockoff}
\label{ako-vs-sko}

Using the same simulation settings as in
Section~\ref{ssec:synthetic-experiment} with \( n = 500, p = 1000 \) and
varying simulation parameters to generate Figure~\ref{fig:fdr-power} in the
main text, we benchmark only aggregation of multiple knockoffs (AKO) with 5 and 10
bootstraps (\( B=5 \) and \( B=10 \)) and compare with simultaneous knockoffs
with 2 and 3 bootstraps (\( \kappa=2 \) and \( \kappa=3 \)).
Results in Figure~\ref{fig:ako-vs-ko-gz} show that while increasing the number
of knockoff bootstraps makes simultaneous knockoffs more conservative, doing so
with AKO makes the algorithm more powerful (and in the worst case retains the
same power with fewer bootstraps).

\subsection{Empirical Evidence on the Independence of Aggregated P-values
  \( \mathbf{\bar{\pi}} \)}
\label{independence-pval}

Using the same simulation settings as in
Section~\ref{ssec:synthetic-experiment} with
\( n = 500, p = 1000, \rho=0.6, \text{snr}=3.0, \text{sparsity}=0.06 \) we
generate 100 observations of \emph{aggregated} p-values \( \bar{\bm\pi} \).
Then, we compute the Spearman rank-order correlation coefficient of the
\emph{Null} \( \bar{\pi}_j \) for these 100 observations along with their
two-sided p-value (for a hypothesis test whose null hypothesis is that two sets
of data are uncorrelated).

The results are illustrated in Figure~\ref{fig:spearman}: the Spearman
correlation values are concentrated around zero, while the distribution of
associated p-values seems to be a mixture between a uniform distribution and a
small mixture component consisting of mostly non-null p-values.
This indicates near independence between the aggregated p-values using
quantile-aggregation \cite{meinshausen_p-values_2009}, hence justifies our use
of BH step-up procedure for selecting FDR controlling threshold in the AKO
algorithm.

\begin{remark}
  Again, it is worth noticing that the empirical evidence we have shown is only
  done in a setting with a Toeplitz structure for the covariance matrix.
  However, as explained in the main text, this correlation setting is usually
  found in neuroimaging and genomics data.
  Hence, we believe that assuming short-distance correlations between the
  $(X_j)_{j \in [p]}$ is a mild assumption, which should be satisfied in the
  practical scenarios where we want to apply our algorithm.
\end{remark}

The decoding maps returned by the KO, AKO and DL inference procedures
are presented in Figure~\ref{fig:hcp-all}.
As quantified by the Jaccard index in the main text, we observe that the AKO
solution is typically closer to an external method based on the desparsified
lasso (DL).
Moreover, AKO is also typically more sensitive than KO alone.

The seven classification problems are the following:
\begin{itemize}
\item Emotion: predict whether the participant watches an angry face or a geometric shape.
\item Gambling: predict whether the participant gains or loses gambles.
\item Motor foot: predict whether the participant moves the left or right foot.
\item Motor hand: predict whether the participant moves the left or right hand.
\item Relational: predict whether the participant matches figures or identified
  feature similarities.
\item Social: predict whether the participant watches a movie with social behavior or not.
\item Working Memory: predict whether the participant does a 0-back or a 2-back task.
\end{itemize}

%%%%%%%%%%%%%%%%%%%%%%%%%%%%%%%%%%%%%%%%%%%%%%%%%%%%%%%%%%%%%%%%%%%%%%%%%%%%%%%%

\begin{figure}[h]
  \centering
  \includegraphics[width=0.7\textwidth]{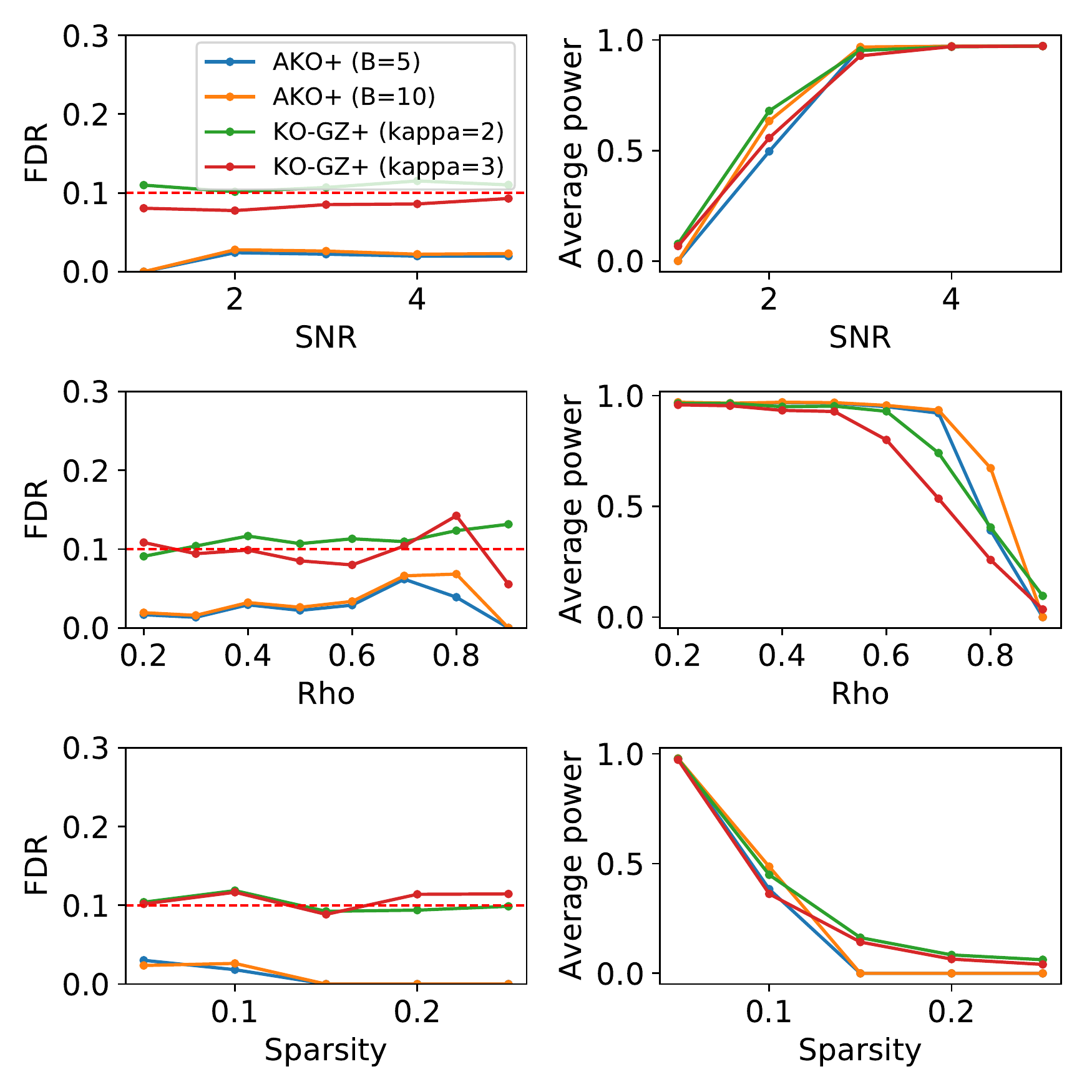}
  \caption {\textbf{Aggregation of multiple knockoffs ($B=5$ and $B=10$)
      vs. simultaneous knockoffs (\( \kappa =2\) and \( \kappa =3\)).}
    A clear loss in statistical power is demonstrated in the latter
    method when increasing the number of bootstraps $\kappa$, while the former (AKO) shows
    the opposite: with $B=10$ bootstraps there are small, yet consistent
    gains in the number of true detections compared to using only $B=5$
    bootstraps.  }
  \label{fig:ako-vs-ko-gz}
\end{figure}

\begin{figure}[h]
  \centering
  \includegraphics[width=0.8\textwidth]{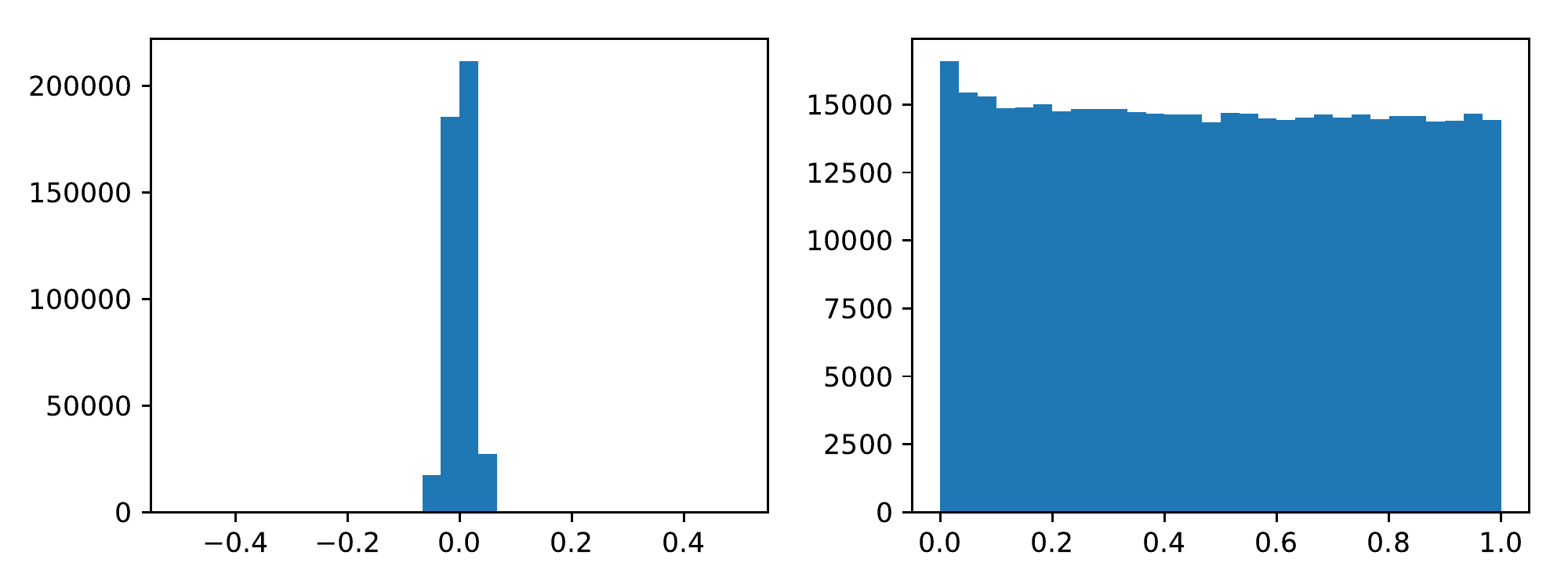}
  \caption {\textbf{Left: Histogram of Spearman correlation values for 100
      samples of null aggregated p-values \( \mathbf{\bar{\pi}_j} \). Right:
      Histogram of corresponding p-values for the Spearman correlation}}
  \label{fig:spearman}
\end{figure}

\begin{figure}[h]
\begin{tabular}{rl}
\raisebox{1cm}{Emotion} & \includegraphics[width=0.6\textwidth]{{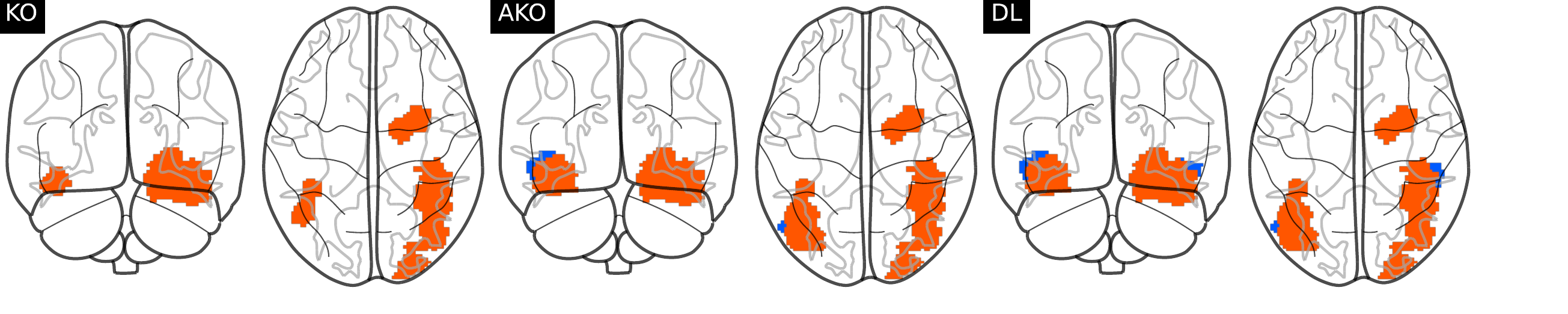}} \\
\raisebox{1cm}{Gambling} & \includegraphics[width=0.6\textwidth]{{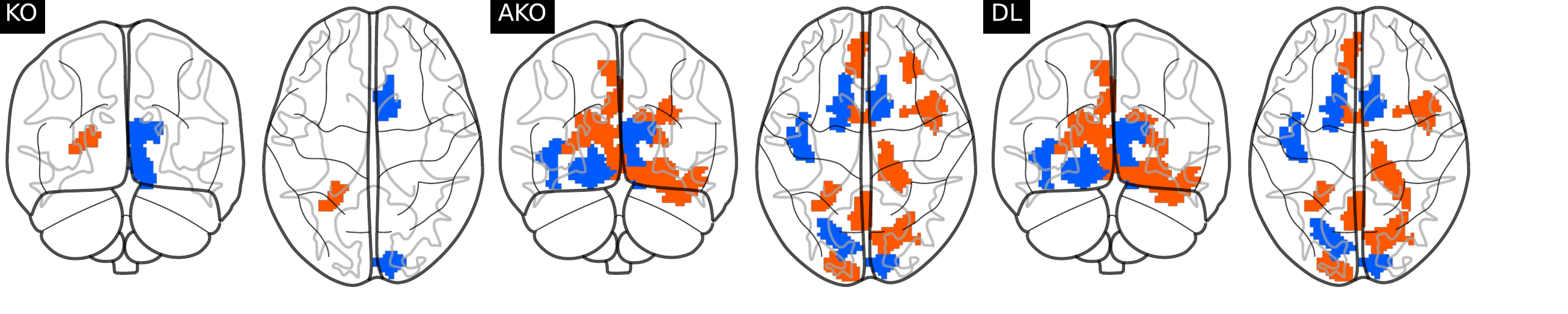}} \\
\raisebox{1cm}{Motor foot}  &\includegraphics[width=0.6\textwidth]{{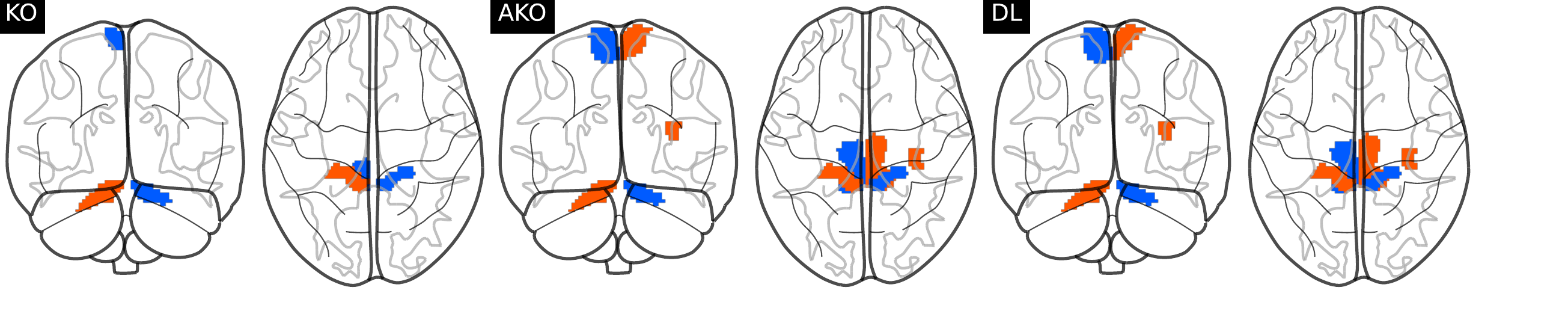}} \\
\raisebox{1cm}{Motor hand} &  \includegraphics[width=0.6\textwidth]{{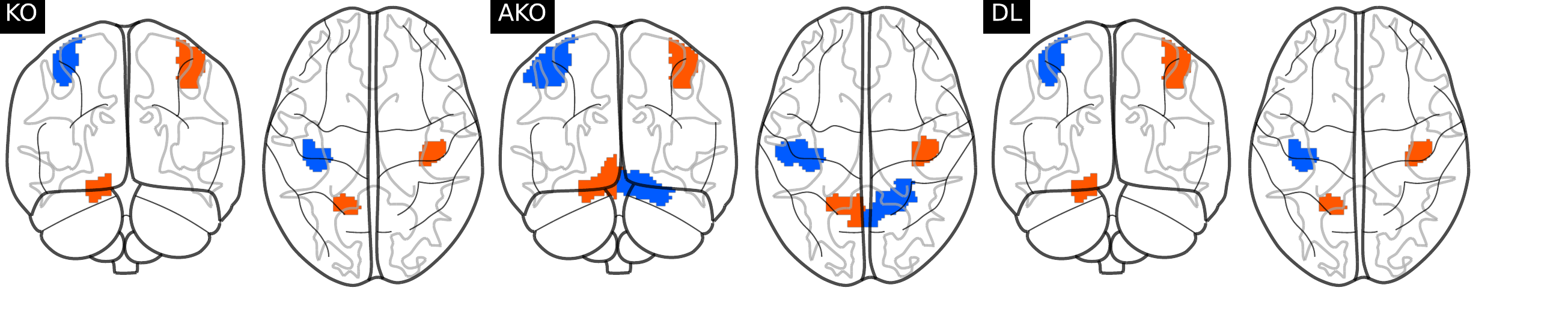}} \\
\raisebox{1cm}{Relational} & \includegraphics[width=0.6\textwidth]{{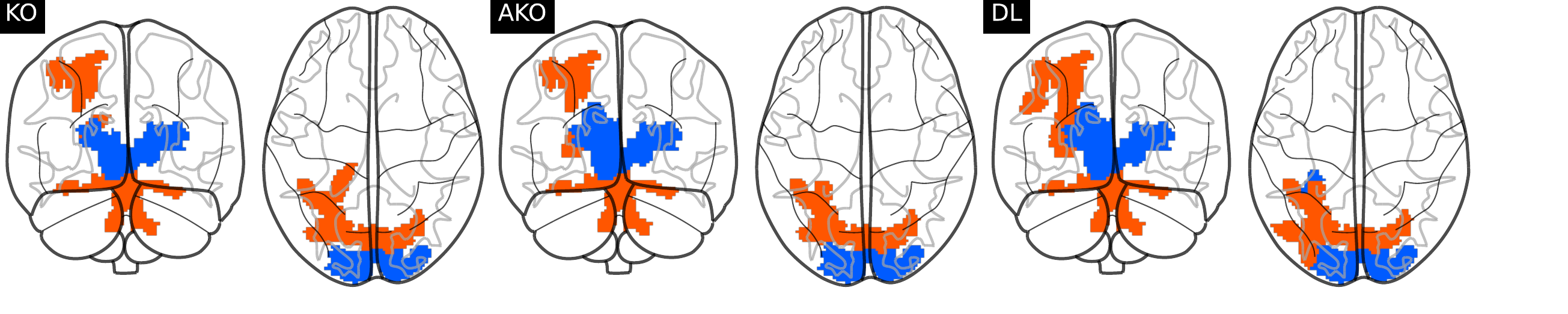}} \\
\raisebox{1cm}{Social} & \includegraphics[width=0.6\textwidth]{{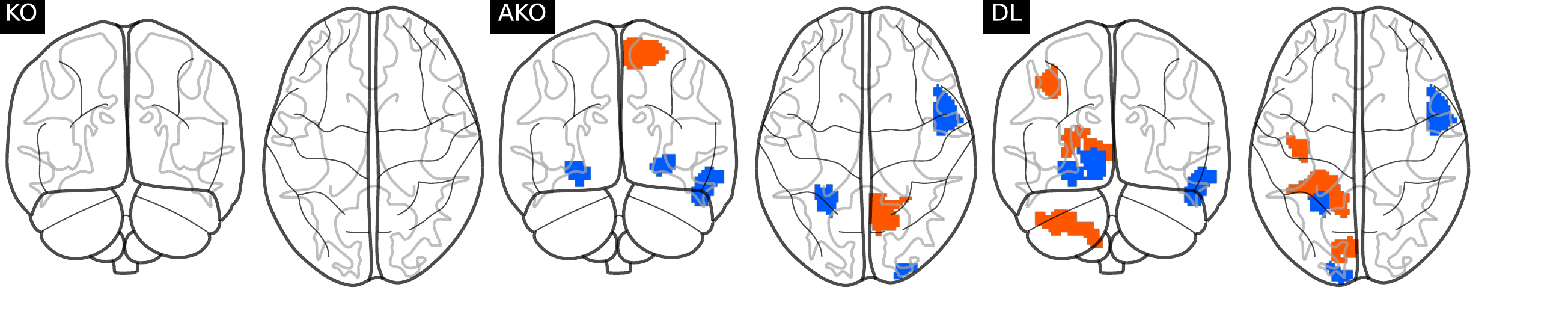}} \\
\raisebox{1cm}{Working memory} & \includegraphics[width=0.6\textwidth]{{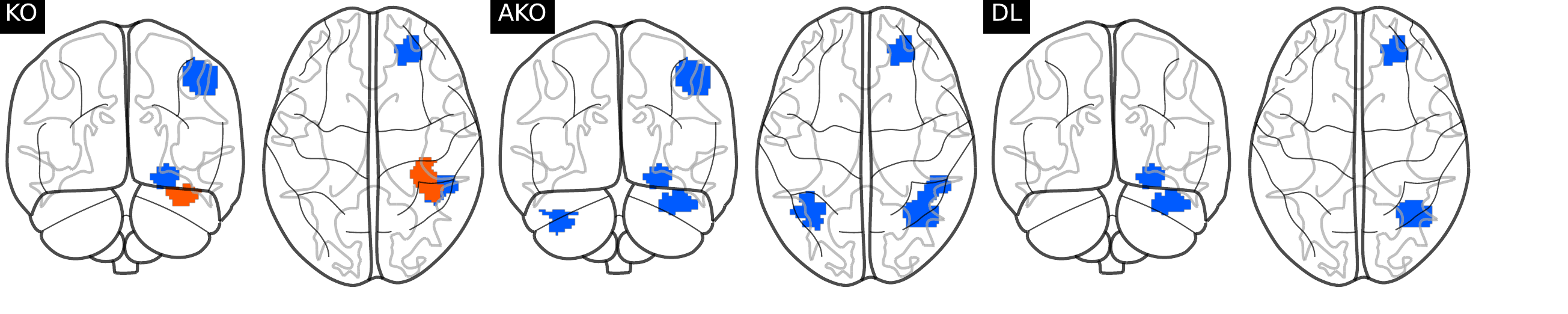}} \\
\end{tabular}
\caption {\textbf{Decoding maps obtained for seven classification tasks.} 
  Emotion, gambling, motor foot, motor hand, relational, social and working
  memory refer to 7 binary tasks that were considered based on the HCP900
  dataset. We observe that AKO is typically more sensitive than KO, and yields
  solution closer to that of an independent solution based on a desparsified-Lasso (DL) estimator.
\label{fig:hcp-all}
}
\end{figure}

%%%%%%%%%%%%%%%%%%%%%%%%%%%%%%%%%%%%%%%%%%%%%%%%%%%%%%%%%%%%%%%%%%%%%%%%%%%%%%%%

\end{document}